\documentclass{amsart}
\usepackage[all]{xy}
\usepackage{amsmath,amsfonts,amssymb,amsthm,epsfig,amscd,comment,latexsym,psfrag}
\usepackage{nicefrac,xspace,tikz}
\usetikzlibrary{arrows}

\newtheorem{Theorem}{Theorem}[section]
\newtheorem{prop}[Theorem]{Proposition}
\newtheorem{lemma}[Theorem]{Lemma}
\newtheorem{cor}[Theorem]{Corollary}
\newtheorem{conj}[Theorem]{Conjecture}

\theoremstyle{definition}
\newtheorem{Remark}[Theorem]{Remark}

\def\bb{({\bf i}, {\bf j})}
\def\Aut{\mbox{Aut}}
\def\ul{{\underline{\ell}}}
\def\T{{\sf{\Sigma}}}
\def\sT{{\sf{T}}}

\def\Br{\mbox{Br}}

\def\Cone{\mbox{Cone}}
\def\Conv{\mbox{Conv}}
\def\D{{\mathcal{D}}}

\def\I{{I}}

\def\Kom{{\mathsf{Kom}}}

\def\1{{\bf{1}}}

\def\E{\mathsf{E}}
\def\F{\mathsf{F}}
\def\P{\mathsf{P}}
\def\Q{\mathsf{Q}}

\def\catC{\mathcal{C}}
\def\la{\langle}
\def\ra{\rangle}
\def\ep{\epsilon}
\def\k{\Bbbk}
\def\K{\mathcal{K}}

\def\sl{\mathfrak{sl}}

\def\H{\mathcal{H}}
\def\oH{\overline{\mathcal{H}}}
\def\h{\widehat{\mathfrak{h}}}
\def\R{{\sf R}}
\def\Z{\mathbb Z}
\def\N{\mathbb N} 
\def\C{\mathbb C}
\def\g{\mathfrak{g}}

\def\l{\lambda}

\def\End{\mathrm{End}}

\def\id{\mathrm{id}}
\def\s{\sigma}

\def\dmod{{\mathrm{-mod}}} 
 
\newcommand{\Hom}{{\rm Hom}}

\newcommand{\Ind}{{\rm{Ind}}}

\psfrag{Qpm}{$Q_{+-}$} \psfrag{Qmp}{$Q_{-+}$}
\psfrag{Ione}{$\quad\mathbf{1}$}
\psfrag{Xiota}{$\iota$}\psfrag{Xiotap}{$\iota'$}
\psfrag{Xgi}{$g_i$}\psfrag{Xgim1}{$g_i^{-1}$}
\psfrag{Xsumi}{\large{$\sum_{i\in I}$}}
\psfrag{Xbeta}{$\overline{\beta}$}
\psfrag{Xgj}{$g_j$} \psfrag{ifinotj}{\large{if $i\not= j$}}
\psfrag{XSnLm}{$S^n_-\otimes \Lambda^m_+$} 
\psfrag{XLmSn}{$\Lambda^m_+\otimes S^n_-$}
\psfrag{XLm1Sn1}{$\Lambda^{m-1}_+\otimes S^{n-1}_-$}
\psfrag{XSumkb}{$-\sum_{b=0}^{k-2} (k-b-1)$}

\title{Braid group actions via categorified Heisenberg complexes}

\begin{document} 
\setcounter{tocdepth}{1}

\author{Sabin Cautis}
\email{scautis@math.columbia.edu}
\address{Department of Mathematics\\ University of Southern California \\ Los Angeles, LA}

\author{Anthony Licata}
\email{anthony.licata@anu.edu.au}
\address{Mathematical Sciences Institute\\ Australian National University \\ Canberra, Australia}

\author{Joshua Sussan}
\email{jsussan@mercy.edu}
\address{Department of Mathematics and Computer Information Sciences \\ Mercy College \\ Dobbs Ferry, NY}


\begin{abstract} 
We construct categorical braid group actions from 2-representations of a Heisenberg algebra. These actions are induced by certain complexes which generalize spherical (Seidel-Thomas) twists and are reminiscent of the Rickard complexes defined by Chuang-Rouquier. Conjecturally, one can relate our complexes to Rickard complexes using categorical vertex operators. 
\end{abstract}

\maketitle

\tableofcontents

\section{Introduction}

Let $\D(0)$ and $\D(1)$ be graded, triangulated categories and let $\P: \D(0) \rightarrow \D(1)$ and $\Q: \D(1) \rightarrow \D(0)$ be bi-adjoint functors up to a grading shift (which we take to be equal to $2$ for convenience). If $\Q \circ \P \cong \1_0 \la -1 \ra \oplus \1_0 \la 1 \ra$, where $\1_0$ denotes the identity functor of $\D(0)$, then $\P$ is called a spherical functor. This notion is due (in various levels of generality) to Seidel-Thomas \cite{ST}, Horja \cite{H}, Anno \cite{A} and Rouquier \cite{Rou1}. 

The general theory of spherical twists states that $\T := \Cone(\P \circ \Q \la -1 \ra \rightarrow \1_1)$ is an autoequivalence of $\D(1)$. One important reason to consider equivalences coming from spherical functors is that if $\{\P_i\}$ is a configuration of spherical functors, one for each node $i$ of a simply laced Kac-Moody Dynkin diagram $D$, then the associated auto equivalences $\T_i$ will define an action of the corresponding braid group $\Br_D$ on $\D(1)$.

The notion of a spherical twist was generalized in \cite{CR,CK2} to that of categorical $\g$ actions, with $\g$ a symmetric Kac-Moody algebra. Such categorical $\g$ actions can be used to construct further examples of braid group actions. Another generalization of spherical twists, which replaces the role of the Kac-Moody algebra by a Heisenberg algebra, is the subject of the current paper. Namely, fix a simply laced Kac-Moody Dynkin diagram $D$ with vertex set $I$. For each $n \ge 0$, suppose we have a collection of additive categories $\D(n)$ together with bi-adjoint functors 
$$\P_i : \D(n) \rightarrow \D(n+1) \ \ \text{ and } \ \ \Q_i: \D(n+1) \rightarrow \D(n)$$ 
for any $i \in I$ which give a 2-representation of a particular Heisenberg algebra (see Section \ref{sec:2heis}). Roughly speaking, this means that we have isomorphisms
\begin{equation}\label{eq:heiseq}
\Q_i \circ \P_i \cong \P_i \circ \Q_i \oplus \1_n \la -1 \ra \oplus \1_n \la 1 \ra,
\end{equation}
along with a precise collection of natural transformations of functors. Such a 2-representation generalizes the notion of a spherical functor since $\P_i$ and $\Q_i$ are spherical functors between $\D(0)$ and $\D(1)$. 

However, the data of a Heisenberg 2-representation contains more than just a spherical functor. For instance, the action by natural transformations includes an action of the symmetric group $S_k$ on the composition $\P_i^k$. This splits $\P_i^k$ into a direct sum of indecomposable functors $\P_i^{(\lambda)}$ corresponding to irreducible representations of $S_k$ ($\Q_i^k$ also splits analogously). We may then form a complex
\begin{equation}\label{eq:intro1}
\T_i \1_n := \left[ \dots \rightarrow \bigoplus_{\l \vdash d} \P_i^{(\l)} \Q_i^{(\l^t)} \la -d \ra \1_n \rightarrow \bigoplus_{\l \vdash d-1} \P_i^{(\l)} \Q_i^{(\l^t)} \la -d+1 \ra \1_n \rightarrow \dots \rightarrow \P_i \Q_i \la -1 \ra \1_n\rightarrow \1_n \right]
\end{equation}
(we have omitted the symbol $\circ$ for composition of functors in the above, as we will do for the remainder of the paper). Theorem \ref{thm:main1} of the current paper states that these complexes define a categorical action of the associated braid group on the homotopy category of each $\D(n)$. In particular, each $\T_i$ defines an equivalence of categories. 

An important example where the setup above holds is the following. Let $A$ be the skew zig-zag algebra (defined in section \ref{sec:zigzag}), which is the quadratic dual of the deformed preprojective algebra of a quiver. For $n \geq 0$, we let $A^{[n]}$ denote the wreath product of $A$ with the group algebra of $S_n$ (by convention, we take $A^{[0]} = \C$). By a formal construction, the braid group action of \cite{KS, HK} on the homotopy category $\Kom(A \dmod)$ by spherical twists lifts to a braid group action on $\Kom(A^{[n]} \dmod)$ for each $n$.  

On the other hand, from the point of view of representation theory of infinite dimensional algebras, it is natural to consider the categories $\Kom(A^{[n]} \dmod)$ together. In particular, in \cite{CL1} we define 2-representations of a Heisenberg algebra on $\oplus_n A^{[n]} \dmod$.   Thus there are two algebraic objects of interest: the braid group action (which is somewhat complicated) and the Heisenberg action (which is simpler). The constructions of the current paper explain precisely the relationship between these two actions. In particular, we prove that integrable 2-representations of the Heisenberg algebra always induce braid group actions. 

There is also a geometric version of this setup, where the algebra $A$ is replaced by a surface $X$, the algebra $A^{[n]}$ is replaced by the Hilbert scheme $X^{[n]}$ of $n$ points on $X$, and the triangulated category $\Kom(A^{[n]} \dmod)$ is replaced by the derived category of coherent sheaves $D(X^{[n]})$. Then,  as studied by Ploog \cite{P}, if a group $G$ acts on $D(X)$ then there is an induced action of $G$ on $D(X^{[n]})$.  In particular, if $G$ is an affine braid group of simply-laced type one can take the surface $X$ to be the ALE space $\widehat{\C^2/\Gamma}$, where $\Gamma\subset SL_2(\C)$ is the finite subgroup associated to the affine quiver by the McKay correspondence.  A 2-representation of the associated Heisenberg algebra on $\oplus_n D(X^{[n]})$ was defined in \cite{CL1}, and Theorem \ref{thm:geom} then describes the relationship between this 2-representation and the associated affine braid group action on $D(X^{[n]})$.

In the remainder of the introduction we will give a more detailed exposition of the content in this paper.

\subsection{Heisenberg actions and braid groups} 

To any simply laced Dynkin diagram $D$ one can associate a quantum Heisenberg algebra algebra $\h$. This algebra has generators $P_i^{(n)}, Q_i^{(n)}$ satisfying relations described in Section \ref{sec:hei}.  A representation $V$ of this algebra breaks up into weight spaces $V = \oplus_{\ell \in \N} V(\ell)$ with 
$$P_i^{(n)}: V(\ell) \rightarrow V(\ell+n) \text{ and } Q_i^{(n)}: V(\ell+n) \rightarrow V(\ell).$$

In \cite{CL1} we define a 2-category $\H$ whose Grothendieck group is isomorphic to $\h$. A 2-representation of $\H$ consists of graded, additive categories $\D(\ell)$ where $\ell \in \N$ and, for any partition $\l$, functors 
$$\P_i^{(\l)}: \D(\ell) \rightarrow \D(\ell+|\l|) \text{ and } \Q_i^{(\l)}: \D(\ell+|\l|) \rightarrow \D(\ell)$$
satisfying various relations described in sections \ref{sec:2cat} and \ref{sec:2rep}.

Now, in the homotopy category $\Kom(\H)$ of $\H$ one can define complexes as in (\ref{eq:intro1}) where the differentials are given by certain explicit 2-morphisms described in section \ref{sec:cpx}. Given a 2-representation, each $\T_i \1_n$ defines an endofunctor of $\Kom(\D(n))$. The following is the main result of this paper. 

\begin{Theorem}\label{thm:main1}
For each $n \in \N$, the map $\sigma_i \mapsto \T_i \1_n$ defines a morphism $\Br(D) \rightarrow \Aut(\Kom(\D(n)))$ where $\Br(D)$ denotes the braid group associated with the Dynkin diagram $D$ and the $\sigma_i$'s are its standard generators. 
\end{Theorem}

Now, if we take $A$ to be the zig-zag algebra from \cite{HK} then, following \cite{CL1}, we can define a 2-representation of $\h$ where $\D(n) = A^{[n]} \dmod$. Theorem \ref{thm:main1} above then gives us a morphism $\Br(D) \rightarrow \Aut(\Kom(A^{[n]} \dmod))$. Applying Theorem \ref{thm:geom} (see also Remark \ref{rem:conv}) this also gives us a morphism $\Br(D) \rightarrow \Aut(D(A^{[n]} \dmod))$. 

When $n=1$ this gives the braid group action of Khovanov-Seidel \cite{KS} via spherical twists. For $n > 1$ we recover the action on $D(A^{[n]} \dmod)$ induced from that on $D(A \dmod)$ (see Theorem \ref{thm:braids}). 

\subsection{Another braid group action}
The complexes $\T_i$ also act on the 2-category $\Kom(\H)$ by conjugation. It would be interesting to describe this braid group action explicitly. 

While we do not address the conjugation action on the entire category $\Kom(\H)$ here, in section \ref{sec:braidH} we define another braid group action on $\Kom(\H)$ and conjecture (Conjecture \ref{conj:intertwiner}) that it agrees with the conjugation action. This additional action is defined explicitly by describing how each generator $\sigma_i^{\pm 1}$ of the braid group acts on the generating 1 and 2-morphisms.  

Although conjecturally related to conjugation by the complexes of Theorem \ref{thm:main1}, section \ref{sec:braidH} is independent of the rest of the paper. The proofs in section \ref{sec:braidH} are postponed until the appendix.  

\subsection{Lie algebra actions and braid groups}

The story above closely parallels (and is directly related to) that of quantum groups. Recall that for any Dynkin diagram $D$ on has the associated quantum group $U_q(\g)$. One can consider 2-representations of $\g$, which consist of various additive categories $\D(\l)$ indexed by weights $\l$ and functors 
$$\E_i^{(r)} \1_\l: \D(\l) \rightarrow \D(\l+r\alpha_i) \text{ and } \1_\l \F_i^{(r)}: \D(\l+r\alpha_i) \rightarrow \D(\l).$$
These functors satisfy certain relations lifting those in $U_q(\g)$. For more details see \cite{KL1,KL2,KL3,Rou2,CLa}. 

In analogy with (\ref{eq:intro1}), one can then define the Rickard complexes
$$\sT_i \1_\l := \left[ \dots \rightarrow \F_i^{(\la \l, \alpha_i \ra + s)} \E_i^{(s)} \la -s \ra \1_\l \rightarrow \dots \rightarrow \F_i^{(\la \l,\alpha_i \ra + 1)} \E_i \la -1 \ra \1_\l \rightarrow \F_i^{(\la \l,\alpha_i \ra)} \1_\l \right].$$
These complexes define a morphism $\Br(D) \rightarrow \Aut(\oplus_\l \Kom(\D(\l)))$ just like the one in Theorem \ref{thm:main1}. See \cite{CR,CKL,CK2} for more details.

The relationship between 2-representations of $\h$ and 2-representations of $\g$ is given by the vertex operator constructions from \cite{CL2}. Thus we expect to have the diagram
$$
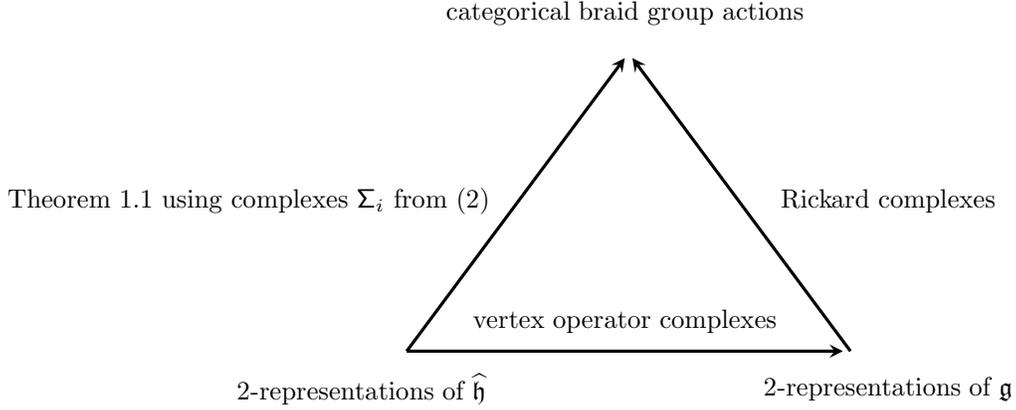
\begin{tikzpicture}[>=stealth]

\draw (0.1,0) -- (5.9,0)[->] [very thick];
\draw (-.5,-.5) node {2-representations of $\h$};
\draw (6.5,-.5) node {2-representations of $\g$};
\draw (3,4.5) node {categorical braid group actions};
\draw (0.1,0) -- (3,3.9)[->] [very thick];
\draw (3.1,3.9) -- (6,0)[<-] [very thick];
\draw (3,.4) node {vertex operator complexes};
\draw (6.5,2) node {Rickard complexes};
\draw (-2,2) node {Theorem \ref{thm:main1} using complexes $\T_i$ from (\ref{eq:intro1})};
\end{tikzpicture}
$$
As the above diagram indicates, we should be able to deduce Theorem \ref{thm:main1} as a consequence of the braid group actions arising from 2-representations of $\g$ \cite{CK2} (this essentially amounts to checking that the diagram above commutes). However, there are several technical details required to give a proof in this way (see section \ref{sec:remarks} for more details), so in this paper we choose to give a direct construction of the left arrow. 

\noindent {\bf Acknowledgments:}
The authors benefited from discussions with Jon Kujawa and Eli Grigsby. S.C. was supported by NSF grant DMS-1101439 and the Alfred P. Sloan foundation. A.L. would like to thank the Institute for Advanced Study for support.

\section{Preliminaries}

We will always work over a base field $\k$ of characteristic zero. 

\subsection{Dynkin data}\label{sec:data}

Let $D$ be a finite graph without edge loops or multiple edges between vertices.  We let $I$ denote the vertex set of $D$, and $E$ the edge set.  The graph $D$ is is the Dynkin diagram of a symmetric simply-laced Kac-Moody algebra. We define a pairing $\la \cdot, \cdot \ra: I \times I \rightarrow \Z$ by $\la i,j \ra := C_{i,j}$ where $C_{i,j}$ is the Cartan matrix associated to our Dynkin diagram. More precisely:
$$\la i,j \ra =
\begin{cases}
2 & \text{ if } i=j \\
-1 & \text{ if } i \ne j \text{ are joined by an edge } \\
0 & \text{ if } i \ne j \text{ are not joined by an edge.}
\end{cases}$$
Associated to this data there is the braid group $\Br(D)$ which is generated by $\{\sigma_i\}_{i \in I}$ subject to the relations $\sigma_i \sigma_j = \sigma_j \sigma_i$ if $i,j \in I$ are not joined by an edge, and $\sigma_i \sigma_j \sigma_i = \sigma_j \sigma_i \sigma_j$ if they are joined.

Fix an orientation $\epsilon$ of $D$. For $i,j\in I$ with $\la i,j \ra = -1$, we set $\epsilon_{ij} = 1$ if the edge is oriented $i \rightarrow j$ by $\epsilon$ and $\epsilon_{ij} = -1$ if oriented $j \rightarrow i$. If $\la i,j \ra = 0$, then we set $\epsilon_{ij} = 0$. Notice that in both cases we have $\epsilon_{ij} = -\epsilon_{ji}$.

\subsection{Partitions}

Let $\l = (\l_1 \ge \l_2 \ge \dots \ge \l_k\geq 0 )$ be a partition.  We denote the size of $\lambda$ by $|\l| := \sum_i \l_i $; we write $\l \vdash n$ if $\l$ is a partition of $n$, and denote the transposed partition by $\l^t$.  If the number of $\l_i = k$ is $a_k$, we also write $\l = (1^{a_1},2^{a_2},\dots,s^{a_s}\dots)$.  For example, in this notation, $(n)^t = (1^n)$.  We write $\l' \subset \l$ if $\l'$ if $\l_i \ge \l'_i$ for all $i$. 

We denote by $\k[S_n]$ the group algebra of the symmetric group and $s_k = (k,k+1) \in S_n$ the simple transposition.  Since the characteristic of $\k$ is 0, $\k[S_n]$ is isomorphic to a direct sum of matrix algebras,
$\k[S_n] = \bigoplus_{\l \vdash n} M_{h_\l}(\k)$.  Here $\{h_\l\}_{\l\vdash n}$ are positive integers, and $M_s(\k)$ is the algebra of $s$-by$s$ matrices over $\k$.  For any partition $\l$ of $n$, we denote by $e_\l \in \k[S_n]$ a minimal idempotent (a matrix unit) in the matrix algebra $M_{h_\l}$ corresponding to $\l$. We denote by $\tau: \k[S_n] \rightarrow \k[S_n]$ the involution which sends $s_i \mapsto -s_i$ for all $i \in I$. The minimal idempotents $e_{\l}$ may be chosen so as to have have $\tau(e_\l) = e_{\l^t}$. 

\subsection{Zig-zag algebras}\label{sec:zigzag}

Let $cD$ denote the doubled quiver, with the same vertex set as $D$ and with two oriented edges (one in each orientation) for each edge of $D$.  Let $\k[dD]$ denote the path algebra of $dD$. A path in $dD$ is described as a sequence of vertices $(i_1 | i_2 | \dots | i_m)$ where $i_k$ and $i_{k+1}$ are connected by an edge in $D$. If $ D $ has more than two nodes then we define $B^D_\ep$ to be the quotient of $\C[dD]$ by the two sided ideal generated by
\begin{itemize}
\item $(a|b|c)$ if $a \ne c$ and
\item $\ep_{ab} (a|b|a) - \ep_{ac} (a|c|a)$ whenever $a$ is connected to both $b$ and $c$.
\end{itemize}
In the above, $e_i$ denotes the constant path which starts and ends at the vertex $i \in I$.
If $D$ consists of the single vertex only, we let $B^D_\ep$ be the algebra generated by $1$ and $X$ with $X^2 = 0$. If $D$ consists of two points joined by a single edge, we deÞne $B^D_\ep$ to be the quotient of $\k[dD]$ by the two-sided ideal spanned by all paths of length greater than two.
Notice that $B^D_\ep$ is $\Z$-graded by the length of the path (we denote the degree of a path by $|\cdot|$). 
The $\Z_2$-grading induced from the $\Z$ grading makes $B^D_\ep$ into a $\Z$-graded superalgebra.

\begin{Remark} The algebras $B^D_\ep$ first appeared in \cite{HK} in the context of categorifying the adjoint representation of the Lie algebra assciated to $D$. 
\end{Remark}

For $n \ge 0$, we define $\Z$-graded superalgebras $B^D_\ep(n) := ({B^D_\ep)}^{\otimes n} \rtimes \k[S_n]$. 
As a vector space, we have $B^D_\ep(n) = ({B^D_\ep)}^{\otimes n} \otimes_\k \k[S_n]$, but for the algebra structure
the tensor product is in the category of superalgebras (see \cite[Section 9.1]{CL1}).  Thus 
$$(a \otimes b) \cdot (a' \otimes b') = (-1)^{|b||a'|}(aa' \otimes bb'),$$
while $S_n$ acts by superpermutations 
$s_k \cdot (b_1 \otimes \dots \otimes b_k \otimes b_{k+1} \otimes \dots \otimes b_n) = (-1)^{|b_k||b_{k+1}|} b_1 \otimes \dots \otimes b_{k+1} \otimes b_k \otimes  \dots \otimes b_n.$
By convention, $B^D_\ep(0) = \k$. To shorten notation we will write 
$$e_{i,m} := (1 \otimes \dots \otimes 1 \otimes e_i \otimes 1 \otimes \dots \otimes 1,1) \in B^D_\ep(n)$$
for the idempotent where $e_i$ occurs in the $m$th tensor factor on the right hand side.  The $\Z$-grading on 
$B^D_\ep(n)$ is induced from that on $B^D_\ep$, with the factor $\k[S_n]$ placed in degree $0$.

\begin{Remark}\label{rem:super}
All the constructions in the remainder of the paper will involve $\Z$-graded superalgebras over $\k$ and graded supermodules or superbimodules over such superalgebras.  For simplicitly, we will write ``algebra", ``module", and "bimodule", omitting the understood prefixes ``$\Z$-graded" and ``super".
\end{Remark}

\subsection{The wreath functor $(\cdot)^{[n]}$}\label{sec:wreath}

If $A$ is a algebra then we can define a new algebra $A^{[n]} := A^{\otimes n} \rtimes \k[S_n]$. The grading and superstructure on $A^{[n]}$ are inherited from that on $A$, with the understanding that $S_n$ acts on $A$ by superpermutations and that the subalgebra $\k[S_n]\subset A^{[n]}$ is in degree 0.  Similarly, if $A_1$ and $A_2$ are algebras and $M$ is an $(A_2,A_1)$-bimodule, then we can define the $(A_2^{[n]}, A_1^{[n]})$-bimodule $M^{[n]} := M^{\otimes n} \rtimes \k[S_n]$. 

To describe $(\cdot)^{[n]}$ as a functor, it is convenient to use the language of 2-categories.  Let $\catC_a$ be the 2-category whose objects are algebras, 1-morphisms are bimodules, and 2-morphisms are bimodule maps.  Composition of 1-morphisms is tensor product of bimodules, and composition of 
2-morphisms is composition of bimodule maps.  

\begin{lemma}\label{lem:[n]functor}
The map $(\cdot) \mapsto (\cdot)^{[n]}$ is defines a 2-functor
$$
	(\cdot)^{[n]}:  \catC_a\longrightarrow \catC_a. 
$$
\end{lemma}
\begin{proof}
For $M_1$ is an $(A_2,A_1)$-bimodule and $M_2$ is an $(A_3,A_2)$ bimodule, 
we define $M_2^{[n]} \otimes_{A_2^{[n]}} M_1^{[n]} \rightarrow (M_2 \otimes_{A_2} M_1)^{[n]}$ by
$$(m_1 \otimes \dots \otimes m_n, \sigma) \otimes (m'_1 \otimes \dots \otimes m'_n, \sigma') \mapsto ((m_1 \otimes m'_{\sigma(1)}) \otimes \dots \otimes (m_n \otimes m'_{\sigma(n)}), \sigma \sigma').$$
It is not difficult to check that this map is an isomorphism.  Thus $(\cdot)\mapsto(\cdot)^{[n]}$ respects composition of 1-morphisms.  It is also clear that $(\cdot)\mapsto (\cdot)^{[n]}$ intertwines compositions of 2-morphisms, for if $f: M_1 \rightarrow M_2$ is a map of bimodules then 
$$f^{[n]} := (f \otimes \dots \otimes f, 1): M_1^{\otimes n} \rtimes \k[S_n] \rightarrow M_2^{\otimes n} \rtimes \k[S_n]$$ 
is a morphism $M_1^{[n]} \rightarrow M_2^{[n]}$ with $(f_2 \circ f_1)^{[n]} = f_2^{[n]} \circ f_1^{[n]}$. 
\end{proof}

The functor $(\cdot)^{[n]}$ is somewhat subtle. In particular,
\begin{itemize}
\item $(\cdot)^{[n]}$ is not linear: if $f,g \in \Hom(M_1,M_2)$, then both $(f+g)^{[n]}$ and $f^{[n]} + g^{[n]}$ are well-defined elements of $\Hom(M_1^{[n]},M_2^{[n]})$ but in general they are not equal to each other.
\item $(\cdot)^{[n]}$ is not additive: in general $(M_1 \oplus M_2)^{[n]}$ and $M_1^{[n]}\oplus M_2^{[n]}$ are not isomorphic (this is already clear at the level of vector spaces via a dimension count). Subsequently, $(\cdot)^{[n]}$ is neither left exact nor right exact.
\end{itemize}

However, $(\cdot)^{[n]}$ behaves well with respect to homotopies of complexes. Suppose 
$M_\bullet = M_0 \rightarrow \dots \rightarrow M_\ell$ is a complex of $(A_1,A_2)$-bimodules. Then the complex $M_\bullet^{[n]}$ is a complex of $(A_1^{[n]},A_2^{[n]})$-bimodules.  The slightly subtle part of this definition is the definition of the boundary map in the complex $M_\bullet^{[n]}$; the easiest way to define it is to consider $M_\bullet$ as a supermodule over  the superalgebra $A \otimes_\k \k[d]/d^2$, where $d$ has superdegree one. Then $M_\bullet^{[n]}$ is naturally an $(A \otimes_\k \k[d]/d^2)^{[n]} \cong A^{[n]}\otimes_{\k[S_n]} (\k[d]/d^2)^{[n]}$ supermodule.  Now the coproduct
$$ \Delta: \k[d]/d^2 \rightarrow (\k[d]/d^2)^{\otimes n}\subset (\k[d]/d^2)^{[n]} $$
given by
$$ \Delta(d) = (1\otimes1 \otimes \hdots \otimes d) + (1 \otimes d \otimes 1 \otimes\hdots \otimes 1) + \hdots + (d\otimes 1\otimes \hdots \otimes 1)$$
embeds $A^{[n]} \otimes_\k \k[d]/d^2$ as a subalgebra of $(A\otimes_\k \k[d]/d^2)^{[n]}$. Thus the $(A\otimes_\k \k[d]/d^2)^{[n]}$ supermodule $M_\bullet^{[n]}$ can be restricted to $A^{[n]}\otimes_\k \k[d]/d^2$, and thus $M_\bullet^{[n]}$ may be considered as a complex of $A^{[n]}$-modules.

An important point to keep in mind is that, because all constructions take place in the category of supermodules, the action of $\k[S_n]$ on an n-fold tensor product of graded vector spaces is via superpermutations.  Thus, spelling this out, the action of the simple transposition $s_i$ on the complex $M_\bullet^{[n]}$ is
$$s_i \cdot (m_1 \otimes \hdots \otimes m_i \otimes m_{i+1} \otimes \hdots m_n) = (-1)^{\deg(m_i) \deg(m_{i+1})} m_1\otimes \hdots \otimes m_{i+1}\otimes m_{i}\otimes \hdots m_n, $$
where $\deg(m_i) = |m_i| + |m_i|_h$ where $|m_i|$ denotes the inner graded degree of $m_i$ and $|m_i|_h$ denotes the homological degree of $m_i$.  

Now, the following lemma shows that the functor $(\cdot)^{[n]}$ behaves well with respect to homotopies. (This is not immediately obvious since $(\cdot)^{[n]}$ is not linear and chain homotopies involve linear combinations of maps.)

\begin{lemma}\label{lem:homotopy}
Let $C_\bullet,D_\bullet$ be complexes of $(A_1,A_2)$-bimodules, and suppose that $f,g: C_\bullet \longrightarrow D_\bullet$ are homotopic maps. Then $f^{[n]},g^{[n]}:C_\bullet^{[n]}\longrightarrow D_\bullet^{[n]}$ are homotopic. 
\end{lemma}
\begin{proof}
By assumption, there exists a chain homotopy $h$ with $f-g = d_{D}h + hd_{C}$.  We set 
$$h' = \sum_{i+j=n-1} (f^{\otimes i} \otimes h \otimes g^{\otimes j}, 1).$$
Then one can check that $f^{[n]}-g^{[n]} = d_{D^{[n]}}h' + h'd_{C^{[n]}}$.
\end{proof}

\subsection{Graded 2-categories}

A graded additive $\k$-linear 2-category $\K$ is a category enriched over graded additive $\k$-linear categories. This means that for any two objects $A,B \in \K$ the Hom category $\Hom_{\K}(A,B)$ is a graded additive $\k$-linear category. Moreover, the composition map $\Hom_{\K}(A,B) \times \Hom_{\K}(B,C) \to \Hom_{\K}(A,C)$ is a graded additive $\k$-linear functor.

{\bf Example.} Suppose $B_n$ is a sequence of graded $\k$-algebras indexed by $n \in \N$. Then one can define a 2-category $\K$ whose objects (0-morphisms) are indexed by $\N$, the 1-morphisms are graded $(B_m,B_n)$-bimodules and the 2-morphisms are maps of graded $(B_m,B_n)$-bimodules.

A graded additive $\k$-linear 2-functor $F: \K \to \K'$ is a (weak) 2-functor that maps the Hom categories $\Hom_{\K}(A,B)$ to $\Hom_{\K'}(FA,FB)$ by additive functors that commute with the auto-equivalence $\la 1 \ra$. 

An additive category $\mathcal{C}$ is said to be idempotent complete when every idempotent 1-morphism splits in $\mathcal{C}$. Similarly, we say that the additive 2-category $\K$ is idempotent complete when the Hom categories $\Hom_{\K}(A,B)$ are idempotent complete for any pair of objects $A, B \in \K$, (so that all idempotent 2-morphisms split). All 2-categories in this paper will be idempotent complete. 

\subsubsection{The homotopy 2-category}
If $\K$ is an additive $\k$-linear 2-category then one can define its homotopy 2-category $\Kom(\K)$ as follows. The objects are the same. The 1-morphisms are unbounded complexes of 1-morphisms in $\K$ while the 2-morphisms are maps of complexes. Two complexes of 1-morphisms are then deemed isomorphic if they are homotopy equivalent.

{\bf Example.} Denote by $\catC_{a}$ the 2-category of algebras.  Then in $\Kom(\catC_{a})$:
\begin{itemize}
\item objects are algebras over $\k$,
\item 1-morphisms from $A$ to $B$ are complexes of $(A,B)$-bimodules,
\item 2-morphisms are chain maps up to homotopy.
\end{itemize}
Combining Lemmas \ref{lem:[n]functor} and \ref{lem:homotopy} implies the following. 

\begin{prop}
For each $n \in \N$, the 2-functor $(\cdot) \mapsto (\cdot)^{[n]}$ defines an endofunctor of the 2-category $\Kom(\catC_{a})$.
\end{prop}

\begin{Remark}
The above endofunctors appeared earlier in ~\cite{K}, which emphasized their relevance for constructing group actions on categories.
\end{Remark}

\subsubsection{Triangulated 2-categories}

A graded triangulated category is a graded category equipped with a triangulated structure where the autoequivalence $\la 1 \ra$ takes exact triangles to exact triangles. We denote the homological shift by $[\cdot]$ where $[1]$ denotes a downward shift by one. 

A graded triangulated $\k$-linear 2-category ${\K'}$ is a category enriched over graded triangulated $\k$-linear categories. This means that for any two objects $A,B \in \K'$ the Hom category $\Hom_{\K'}(A,B)$ is a graded additive $\k$-linear triangulated category. 

{\bf Example.} If $\K$ is a $\k$-linear 2-category then $\Kom(\K)$ is a triangulated 2-category. In the remainder of the paper this extra triangulated structure of $\Kom(\K)$ will not play a role and will usually be ignored. 

\section{Quantum Heisenberg algebras}\label{sec:hei}

Here we recall the quantum Heisenberg algebra $\h$ and its Fock space representation. We will denote the quantum integer by 
$$[n] := t^{-n+1} + t^{-n+3} + \dots + t^{n-3} + t^{n-1}.$$

The traditional presentation for the quantum Heisenberg algebra is as a unital algebra generated by $a_i(n)$, where $i \in \I$ and $n \in \Z \setminus \{0\}$.  The relations are
\begin{equation}\label{rel:as}
a_i(m) a_j(n) - a_j(n)a_i(m) = \delta_{m,-n} [n \la i,j \ra] \frac{[n]}{n}.
\end{equation}
When $q=1$, this presentation specializes to the standard presentation of the non-quantum Heisenberg algebra.  

For our purposes, a more convenient presentation of $\h$ takes as generators $\{P_i^{(n)},Q_i^{(n)} \}_{i \in \I, n \geq 0}$ subject to the following relations:
\begin{eqnarray*}
P_i^{(n)} P_j^{(m)} &=& P_j^{(m)} P_i^{(n)} \text{ and } Q_i^{(n)}Q_j^{(m)} = Q_j^{(m)} Q_i^{(n)} \text{ for all } i,j \in \I, \\
Q_i^{(n)} P_j^{(m)} &=&
\begin{cases}
\sum_{k \ge 0} [k+1] P_i^{(m-k)} Q_i^{(n-k)} & \text{ if } i=j, \\
P_j^{(m)} Q_i^{(n)} + P_j^{(m-1)} Q_i^{(n-1)} & \text{ if } \la i,j \ra = -1 \\
P_j^{(m)} Q_i^{(n)} & \text{ if } \la i,j \ra = 0.
\end{cases}
\end{eqnarray*}
By convention $P_i^{(0)} = Q_j^{(0)} = 1$ and $P_i^{(k)} = Q_i^{(k)} = 0$ when $k < 0$ so the summations in the relations above are all finite. Notice that $\h$ has a natural $\Z$-grading where $\deg P_i^{(n)} = n$ and $\deg Q_i^{(n)} = -n$.  

An explicit isomorphism between these two presentations is given in \cite{CL1} and \cite[Section 3.1]{CL2}.

\subsection{The Fock space}\label{sec:fock}
Let $\h^- \subset \h$ denote the subalgebra generated $\{ Q_i^{(n)} \}_{i \in I, n \ge 0}$. Let $\mbox{triv}_0$ denote the trivial (one-dimensional) representation of $\h^-$, where all $Q_i^{(n)}$ ($n>0$) act by zero. Then
$V_{Fock} := \Ind_{\h^-}^{\h}(\mbox{triv}_0)$
is called the Fock space representation of $\h$.

The Fock space has a basis given by elements of the form $P_{i_k}^{(n_k)} \dots P_{i_1}^{(n_1)}(v)$ where $v$ is a vector spanning $\mbox{triv}_0$. This gives a decomposition $V_{Fock} = \oplus_{n \ge 0} V_{Fock}(n)$. To simplify notation, we will denote $P_{i_k}^{(n_k)} \dots P_{i_1}^{(n_1)}(v)$ by $P_{i_k}^{(n_k)} \dots P_{i_1}^{(n_1)}$. 

For any partition $\l \vdash n$ on can define $P_i^{(\l)}$ using Giambelli's formula as the determinant
$$ [P_i^{\lambda}] = \mbox{det}_{kl} [P_i^{(\l_k + l-k)}].$$
For example, $P_i^{(1^2)} = P_iP_i - P_i^{(2)}$. See \cite[Section 7]{CL1} for more details. 

\subsection{The braid group action on $V_{Fock}$}\label{subsec:braidfock}

We now describe a braid group action on $V_{Fock}$. Since $V_{Fock}$ is multiplicatively generated by elements $P_i^{(n)}$ it suffices to describe this action on these generators and extend multiplicatively. On generators, the action is given by 
\begin{equation*}
\sigma_i(P_j^{(n)}) = 
\begin{cases} 
(-t^{-2})^n P_i^{(1^n)} & \text{ if } i=j \\
\sum_{k=0}^n (-t^{-1})^{n-k} P_j^{(k)} P_i^{(n-k)} & \text{ if } \la i,j \ra = -1 \\
P_j^{(n)} & \text{ if } \la i,j \ra = 0  
\end{cases}
\end{equation*}
and
\begin{equation*}
\sigma_i^{-1}(P_j^{(n)}) = 
\begin{cases}
(-t^{2})^n P_i^{(1^n)} & \text{ if } i=j \\
\sum_{k=0}^n (-t)^{n-k} P_j^{(k)} P_i^{(n-k)} & \text{ if } \la i,j \ra = -1 \\
P_j^{(n)} & \text{ if } \la i,j \ra = 0.
\end{cases}
\end{equation*}

\begin{prop}\label{prop:fockspace}
The endomorphisms $ \sigma_i $ and $ \sigma_i^{-1} $  for $ i \in I$ define a representation of the braid group $\Br(D)$ on each weight space $V_{Fock}(n)$ of the Fock space. 
\end{prop}
\begin{proof}
This is a consequence of Proposition \ref{prop:2} and Remark \ref{rem:2} following it. 
\end{proof}

It is already interesting to see this braid action on various basis vectors, as in the example below. 

\begin{lemma}\label{lemma1} Suppose $i,j,k \in I$ are different with $\la i,j \ra = -1 = \la j,k \ra$ and $\la i,k \ra = 0$. Then 
\begin{eqnarray*}
\sigma_i \sigma_j \sigma_i(P_i^{(n)}) &=& t^{-3n} P_j^{(1^n)} \\
\sigma_j \sigma_i \sigma_j(P_k^{(n)}) &=& \sum_{a=0}^n \sum_{b=0}^{n-a} (-1)^b t^{-2(n-a)+b} P_{k}^{(a)} P_{j}^{(b)} P_{i}^{(n-a-b)}
\end{eqnarray*}
\end{lemma}
\begin{proof}
By definition we have 
$$\sigma_i \sigma_j \sigma_i (P_i^{(n)}) = (-t^{-2})^n \sum_{a=0}^n \sum_{b=0}^{n-a} (-1)^{a+n} t^{-2n+b} P_j^{(1^b)} P_i^{(a)} P_i^{(1^{n-a-b})}.$$
This simplifies to give $t^{-3n} P_j^{(1^n)}$ if we use the identity $P_i^{(m)} P_i^{(1^n)} = P_i^{(m,1^n)} + P_i^{(m+1,1^{n-1})}$ (see for instance \cite[Prop. 1]{CL2}). The second identity is similar but more involved so we omit the proof. 
\end{proof}

The action of $ \sigma_i^{\pm 1} $ on the usual generators $ a_j(-n) $ of the Fock space has a somewhat easier description. The proof is a straightforward calculation using the generating functions in \cite[Sect. 2.2.1]{CL1}.

\begin{equation*}
\sigma_i(a_j(-n)) = 
\begin{cases} 
-t^{-2n} a_i(-n) & \text{ if } i=j \\
a_j(-n)+(-1)^nt^{-n} a_i(-n) & \text{ if } \la i,j \ra = -1 \\
a_j(-n) & \text{ if } \la i,j \ra = 0  
\end{cases}
\end{equation*}
and
\begin{equation*}
\sigma_i^{-1}(a_j(-n)) = 
\begin{cases} 
-t^{2n} a_i(-n) & \text{ if } i=j \\
a_j(-n)+(-1)^nt^{n} a_i(-n) & \text{ if } \la i,j \ra = -1 \\
a_j(-n) & \text{ if } \la i,j \ra = 0.  
\end{cases}
\end{equation*}

\section{2-representations of $\h$ and the braid complex}\label{sec:2heis}

In this section we review some facts about 2-representations of $\h$ and define the braid complex $\T_i \1_n$.

\subsection{The 2-category $\H$}\label{sec:2cat}

In \cite{CL1} we introduced a 2-category $\H^\Gamma$ associated to any finite subgroup $\Gamma \subset SL_2(\C)$. In our current language this 2-category is associated to the pair $(D,\ep)$ where $D$ is the affine Dynkin diagram corresponding to $\Gamma$ by the McKay correspondence and $\ep$ is an appropriately chosen orientation of $D$.

That definition generalizes with no effort to give a 2-category $\H^D_\ep$ associated to any simply laced Dynkin diagram $D$ and orientation $\ep$. More precisely, $\H^D_\ep$ is the (idempotent closure of) the additive, graded, $\k$-linear 2-category where
\begin{itemize}
\item 0-morphisms (objects) are indexed by the integers $\Z$,
\item 1-morphisms consist of the identity 1-morphisms $\1_n$ of $n \in \Z$ and compositions, direct sums and grading shifts of $\P_i \1_n: n \rightarrow n+1$ and $\1_n \Q_i: n+1 \rightarrow n$ for $i \in I$,
\item The 2-morphisms are generated by adjunction maps, making $\P_i$ and $\Q_i$ bi-adjoint up to shift, together with certain maps $X_i^j\in \Hom(\P_i,\P_j)$ and $T_{ij}\in \Hom(\P_j\P_i,\P_i\P_j)$ satisfying a series of relations. 
\end{itemize}

The generating 2-morphisms and their relations were described diagrammatically in \cite{CL1} and reviewed in \cite[Section 3.2]{CL2}. In the interest of space, we have elected not to spell them out again. However, the essential algebraic structure of these 2-morphisms is straightforward to summarize, and is enough for the purposes of the current paper:
\begin{itemize}
\item The 1-morphisms $\P_i$ and $\Q_i$ are left and right adjoint to one another, up to a grading shift.
\item For each $n\geq 0$, there is a natural injective map
$$B^D_\ep(n) \longrightarrow \End((\bigoplus_{i \in I} \P_i)^n).$$
In particular, for each $i\in I$ there is an embedding $\k[S_n]\longrightarrow \End(\P_i^n)$.  Since $\H^D_\ep$ is idempotent complete, for any partition $\l\vdash n$ we can define the 1-morphism $\P_i^{(\l)}$ as the image of the idempotent $e_\l \in \k[S_{|\l|}]$ acting on $\P_i^{|\l|}$.
\end{itemize}

From hereon we will fix $D$ and $\ep$ and denote $\H_\ep^D$ simply by $\H$ in order to simplify notation. 

\subsection{2-representation of $\h$}\label{sec:2rep}
A {\em 2-representation of $\h$} consists of a graded, idempotent complete $\k$-linear category $\K$ where 
\begin{itemize}
\item 0-morphisms are graded, $\k$-linear, additive categories $\D(n)$,
\item 1-morphisms are (certain types of) functors between these categories,
\item 2-morphisms are natural transformations of these functors
\end{itemize}
together with a 2-functor $\H \rightarrow \K$. We also require that the space of 2-morphisms between any two 1-morphisms in $\K$ be finite dimensional and that $\Hom_\K(\1_n, \1_n \la \ell \ra)$ is zero if $\ell < 0$ and one-dimensinal if $\ell=0$. 

The fact that the space of maps between any two 1-morphisms is finite dimensional means that the Krull-Schmidt property holds. Thus any 1-morphism has a unique direct sum decomposition (see section 2.2 of \cite{Rin}). Note that if $\K$ satisfies the Krull-Schmidt property then so does $\Kom(\K)$. 

A 2-representation of $\h$ is said to be \emph{integrable} if $\1_n=0$ are zero for $n \ll 0$. For example, in \cite{CL1} we constructed an integrable 2-representation where $n$ corresponded to the category of coherent sheaves on the Hilbert scheme of points $\mbox{Hilb}^n(\widehat{\C^2/\Gamma})$ where $\Gamma \subset SL_2(\C)$ is the finite subgroup associated to our Dynkin diagram. 

\subsection{The 2-representation $\K_{Fock}$}

The 2-representation $\K_{Fock}$ categorifies the Fock space representation $V_{Fock}$. It consists of 
\begin{itemize}
\item 0-morphisms: $n$ indexes the category of projective $B^D_\ep(n)$-modules.
\item 1-morphisms: $(B^D_\ep(n), B^D_\ep(n'))$-bimodules which are direct summands of tensor products of the bimodules $\P_i \1_n$ and $\1_n \Q_i$. 
\item 2-morphisms: all maps of bimodules. 
\end{itemize}
The 2-functor $\H \rightarrow \K_{Fock}$ was defined in \cite[Section 9]{CL1}. 

\subsection{Some technical facts}\label{sec:technical}

We gather some useful facts about 2-representations of $\h$. 

\begin{lemma}\label{lem:1}
For an arbitrary partition $\l$ we have
$$\Q_i^{(\l)} \P_i \cong \P_i \Q_i^{(\l)} \bigoplus_{\l' \subset \l} \Q_i^{(\l')} \la -1,1 \ra \text{ and } \Q_i \P_i^{(\l)} \cong \P_i^{(\l)} \Q_i \bigoplus_{\l' \subset \l} \P_i^{(\l')} \la -1,1 \ra$$
where the sums are over all $\l' \subset \l$ with $|\l'| = |\l|-1$. 
\end{lemma}
\begin{proof}
Lemma 3.3 \cite{CL2}. 
\end{proof}

\begin{lemma}\label{lem:2}
Suppose $\l,\l',\mu$ and $\mu'$ are partitions such that $|\l| > |\l'|$ and $|\mu|>|\mu'|$. Then $\dim \Hom( \P_i^{(\l)} \Q_i^{(\mu)}, \P_i^{(\l')} \Q_i^{(\mu')} \la 1 \ra) \le 1$ with equality if and only if $\l' \subset \l$ and $\mu' \subset \mu$ with $|\l| = |\l'|+1$ and $|\mu|=|\mu'|+1$. In this case, the map is spanned by the composition 
$$\P_i^{(\l)} \Q_i^{(\mu)} \longrightarrow \P_i^{(\l')} \P_i \Q_i \Q_i^{(\mu')} \longrightarrow \P_i^{(\l')} \Q_i^{(\mu')} \la 1 \ra$$
where the second map is given by adjunction. 
\end{lemma}
\begin{proof}
Lemma 3.4 \cite{CL2}. 
\end{proof}

\begin{lemma}\label{lem:3}
Consider partitions $\l,\l',\mu,\mu'$ such that $|\l|-|\l'|=2=|\mu|-|\mu'|$. Then 
\begin{equation}\label{eq:hom}
\Hom(\P_i^{(\l)} \Q_i^{(\mu)}, \P_i^{(\l')} \Q_i^{(\mu')} \la 2 \ra)
\end{equation}
is zero unless $\l' \subset \l$ and $\mu' \subset \mu$ in which case its dimension is equal to 
$$\begin{cases}
2 \hspace{.5cm} \text{ if } \l \setminus \l' \text{ and } \mu \setminus \mu' \text{ both consist of two boxes in different rows and columns, } \\ 
0 \hspace{.5cm} \text{ if } \l \setminus \l' \text{ consists of two boxes in same row (resp. column) } \\
\hspace{.6cm} \text{ while } \mu \setminus \mu' \text{ consists of two boxes in same column (resp. row), } \\
1 \hspace{.5cm} \text{ otherwise. } 
\end{cases}$$
\end{lemma}
\begin{proof}
First one notes that a map in (\ref{eq:hom}) consists of a composition of two degree $1$ maps.  Since the only degree $1$ maps involving only vertex $i$ are induced from the adjunction map $\P_i\Q_i \rightarrow \id\la 1 \ra$, if follows that any map in (\ref{eq:hom}) factors 
$$
	\P_i^{(\l)} \Q_i^{(\mu)} \rightarrow \P_i^{|\l|-1}\Q_i^{|\mu|-1} \la 1 \ra \rightarrow \P_i^{(\l')} \Q_i^{(\mu')} \la 2 \ra.
$$
Since $\P_i^{|\l|-1}\Q_i^{|\mu|-1} \la 1 \ra$ is isomorphic to  a direct sum of indecomposable terms $\P_i^{(\l'')} \Q_i^{(\mu'')} \la 1 \ra$ for partitions $\l''$, $\mu''$, we see that  the space of maps in (\ref{eq:hom}) is spanned by maps which factor through some $\P_i^{(\l'')} \Q_i^{(\mu'')} \la 1 \ra$.  Subsequently, by Lemma \ref{lem:2}, it follows that (\ref{eq:hom}) is zero unless $\l' \subset \l$ and $\mu' \subset \mu$. 

Now, we proceed by induction on $|\l|+|\mu|$. We prove the first case above, the others follow similarly. 

First, suppose that there exists some $\l' \not\subset \nu \subset \l$. Then $\P_i \P_i^{(\nu)} \cong \P_i^{(\l)} \oplus_{\gamma} \P_i^{(\gamma)}$. Subsequently 
$$\Hom(\P_i^{(\l)} \Q_i^{(\mu)}, \P_i^{(\l')} \Q_i^{(\mu')} \la 2 \ra) \cong \Hom(\P_i \P_i^{(\nu)} \Q_i^{(\mu)}, \P_i^{(\l')} \Q_i^{(\mu')} \la 2 \ra)$$
since, for all $\gamma$ in the sum above, $\Hom(\P_i^{(\gamma)} \Q_i^{(\mu)}, \P_i^{(\l')} \Q_i^{(\mu')} \la 2 \ra) = 0$ because $\l' \not\subset \gamma$. Hence
\begin{eqnarray*}
(\ref{eq:hom}) 
&\cong& \Hom(\P_i^{(\nu)} \Q_i^{(\mu)}, \Q_i \P_i^{(\l')} \Q_i^{(\mu')} \la 1 \ra) \\
&\cong& \Hom(\P_i^{(\nu)} \Q_i^{(\mu)}, \P_i^{(\l')} \Q_i \Q_i^{(\mu')} \la 1 \ra) \oplus_{\rho \subset \l'} \Hom(\P_i^{(\nu)} \Q_i^{(\mu)}, \P_i^{(\rho)} \Q_i^{(\mu')} \la 0,2 \ra). 
\end{eqnarray*}
The left term above is zero since $\l' \not\subset \nu$. Using induction, every term in the summation on the right is also zero with the exception of the $\rho$ which satisfies $\rho \subset \nu$. The result now follows by induction. 

Now, suppose that no such $\nu$ as above exists. Thus means that $\l$ is a partition whose Young diagram is a union of at most two rectangles, as illustrated below:
$$
\begin{tikzpicture}[>=stealth]
\draw (-1,-.5) node {$\lambda =$};

\draw (0,0) -- (2,0)[very thick];
\draw (2,0) -- (2,-.5)[very thick];
\draw (0,0) -- (0,-1)[very thick];
\draw (0,-1) -- (1,-1)[very thick];
\draw (1,-.5) -- (1,-1)[very thick];
\draw (1,-.5) -- (2,-.5)[very thick];
\end{tikzpicture}
$$
Now, we choose $\nu \subset \l$ but this time $\l' \subset \nu$ so that once again we have $\P_i \P_i^{(\nu)} \cong \P_i^{(\l)} \oplus_{\gamma} \P_i^{(\gamma)}$. On the one hand we have 
\begin{eqnarray*}
\Hom(\P_i \P_i^{(\nu)} \Q_i^{(\mu)}, \P_i^{(\l')} \Q_i^{(\mu')} \la 2 \ra) 
&\cong& (\ref{eq:hom}) \oplus_\gamma \Hom(\P_i^{(\gamma)} \Q_i^{(\mu)}, \P_i^{(\l')} \Q_i^{(\mu')} \la 2 \ra) \\
&\cong& (\ref{eq:hom}) \oplus \k^{2 \# \{\gamma\}}
\end{eqnarray*}
where the second line follows by applying the first step above and then induction. On the other hand, by adjunction
\begin{eqnarray*}
& & \Hom(\P_i \P_i^{(\nu)} \Q_i^{(\mu)}, \P_i^{(\l')} \Q_i^{(\mu')} \la 2 \ra) \\
&\cong& \Hom(\P_i^{(\nu)} \Q_i^{(\mu)}, \Q_i \P_i^{(\l')} \Q_i^{(\mu')} \la 1 \ra) \\
&\cong& \Hom(\P_i^{(\nu)} \Q_i^{(\mu)}, \P_i^{(\l')} \Q_i \Q_i^{(\mu')} \la 1 \ra) \oplus_{\rho \subset \l'} \Hom(\P_i^{(\nu)} \Q_i^{(\mu)}, \P_i^{(\rho)} \Q_i^{(\mu')} \la 0,2 \ra) \\
&\cong& (\k \oplus \k) \oplus (\k^{2(\# \{\gamma\} - 1)} \oplus \k \oplus \k) 
\end{eqnarray*}
where the last line requires a quick case by case analysis of the possible $\rho$'s. Comparing these two expressions gives us that $(\ref{eq:hom}) \cong \k^2$ and the induction is complete. 
\end{proof}

\begin{cor}\label{cor:3}
Suppose $\nu \subset \mu \subset \l$ are partitions with $|\l| = |\mu|+1 = |\nu|+2$ such that $\l \setminus \nu$ consists of two boxes which are not in the same row or column. Then the composition
\begin{equation}\label{eq:cor3}
\P_i^{(\l)} \Q_i^{(\l^t)} \rightarrow \P_i^{(\mu)} \Q_i^{(\mu^t)} \la 1 \ra \rightarrow \P_i^{(\nu)} \Q_i^{(\nu^t)} \la 2 \ra
\end{equation}
consisting of maps from Lemma \ref{lem:2} is nonzero. 
\end{cor}
\begin{proof}
Fix $\l,\nu$ as above. Any map $\P_i^{(\l)} \Q_i^{(\l^t)} \rightarrow \P_i^{(\nu)} \Q_i^{(\nu^t)} \la 2 \ra$ must factor as 
$$\P_i^{(\l)} \Q_i^{(\l^t)} \rightarrow \P_i^{(\nu)} \P_i \P_i \Q_i \Q_i \Q_i^{(\nu^t)} \rightarrow \P_i^{(\nu)} \Q_i^{(\nu^t)} \la 2 \ra$$
where the right most map is given by two adjunctions. Now, since
$$\Hom(\P_i^{(2)} \Q_i^{(1^2)},  \1 \la 2 \ra) = 0 = \Hom(\P_i^{(1^2)} \Q_i^{(2)},  \1 \la 2 \ra)$$
this map factors as
$$\P_i^{(\l)} \Q_i^{(\l^t)} \rightarrow \P_i^{(\nu)} \P_i^{(2)} \Q_i^{(2)} \Q_i^{(\nu^t)} \oplus \P_i^{(\nu)} \P_i^{(1^2)} \Q_i^{(1^2)} \Q_i^{(\nu^t)} \rightarrow \P_i^{(\nu)} \Q_i^{(\nu^t)} \la 2 \ra.$$ 
Now, $\Hom(\P_i^{(\l)} \Q_i^{(\l^t)}, \P_i^{(\nu)} \P_i^{(2)} \Q_i^{(2)} \Q_i^{(\nu^t)}) \cong \k \cong \Hom(\P_i^{(\l)} \Q_i^{(\l^t)}, \P_i^{(\nu)} \P_i^{(1^2)} \Q_i^{(1^2)} \Q_i^{(\nu^t)})$ and, by Lemma \ref{lem:3}, $\dim_\k \Hom(\P_i^{(\l)} \Q_i^{(\l^t)}, \P_i^{(\nu)} \Q_i^{(\nu^t)} \la 2 \ra) = 2$. Hence the compositions 
\begin{eqnarray*}
&& \P_i^{(\l)} \Q_i^{(\l^t)} \rightarrow \P_i^{(\nu)} \P_i^{(2)} \Q_i^{(2)} \Q_i^{(\nu^t)} \rightarrow \P_i^{(\nu)} \Q_i^{(\nu^t)} \la 2 \ra \\
&& \P_i^{(\l)} \Q_i^{(\l^t)} \rightarrow \P_i^{(\nu)} \P_i^{(1^2)} \Q_i^{(1^2)} \Q_i^{(\nu^t)} \rightarrow \P_i^{(\nu)} \Q_i^{(\nu^t)} \la 2 \ra
\end{eqnarray*}
are both nonzero. Fortunately, the unique map $\P_i^{(\l)} \rightarrow \P_i^{(\nu)} \P_i^{(2)}$ factors through $\P_i^{(\mu)} \P_i$ and likewise $\Q_i^{(\l^t)} \rightarrow \Q_i^{(2)} \Q_i^{(\nu^t)}$ factors through $\Q_i \Q_i^{(\mu)}$ for any $\nu \subset \mu \subset \l$. Hence 
$$\P_i^{(\l)} \Q_i^{(\l^t)} \rightarrow \P_i^{(\mu)} \P_i \Q_i \Q_i^{(\mu^t)} \rightarrow \P_i^{(\nu)} \P_i \P_i \Q_i \Q_i \Q_i^{(\nu^t)} \rightarrow \P_i^{(\nu)} \Q_i^{(\nu^t)} \la 2 \ra$$
is nonzero and the result follows since this composition is the same as the one in (\ref{eq:cor3}). 
\end{proof}

\subsection{The braid complex $\T_i \1_n$}\label{sec:cpx}

Suppose $\K$ is some integrable 2-representation of $\h$. The main object of study in this paper is the following complex of 1-morphisms 
\begin{equation}\label{eq:cpx}
\T_i \1_n := \left[ \dots \rightarrow \bigoplus_{\l \vdash d} \P_i^{(\l)} \Q_i^{(\l^t)} \la -d \ra \1_n \rightarrow \bigoplus_{\l \vdash d-1} \P_i^{(\l)} \Q_i^{(\l^t)} \la -d+1 \ra \1_n \rightarrow \dots \rightarrow \P_i \Q_i \la -1 \ra \1_n\rightarrow \1_n \right].
\end{equation}
which lives naturally in $\Kom(\K)$. The right hand term $\1_n$ of this complex is in cohomological degree zero. Notice that since $\K$ is integrable this complex is finite. 

The differential in (\ref{eq:cpx}) is defined as the composition 
\begin{equation}\label{eq:diff}
\P_i^{(\l)} \Q_i^{(\l^t)} \rightarrow \P_i^{(\mu)} \P_i \Q_i \Q_i^{(\mu^t)} \rightarrow \P_i^{(\mu)} \Q_i^{(\mu^t)} \la 1 \ra
\end{equation}
where the first map is inclusion and the second is given by adjunction (note that we must have $\mu \subset  \l$ in order for this map to be nonzero). The inclusion map is unique but only up to multiple. Likewise, by Lemma \ref{lem:2}, the composition is unique but only up to multiple. Fortunately, Proposition \ref{prop:Tcpx} and Remark \ref{rem:1} below shows that there is a unique way (up to homotopy) to choose these multiples in order to get an indecomposable complex. 

\begin{Remark} 
Note that we did not check directly that the compositions in (\ref{eq:diff}) define a differential ({\it i.e.} square to zero). It is possible to check this directly by generalizing the statement in Lemma \ref{lem:3} but we avoid doing this extra work because later we will conclude for free that there exists an indecomposable complex with terms as in (\ref{eq:cpx}). Then Proposition \ref{prop:Tcpx} will tell us that the differentials are indeed given by (\ref{eq:diff}). 
\end{Remark}

\begin{prop}\label{prop:Tcpx}
Any complex whose terms are the same as those of $\T_i$ and which is indecomposable in $\Kom(\K)$ is homotopic to $\T_i$. 
\end{prop}
\begin{proof}
Let us fix $n \in \Z$ and look at $\T_i \1_n$. For any $\l,\mu$ such that $\Hom(\P_i^{(\l)} \Q_i^{(\l^t)} \1_n, \P_i^{(\mu)} \Q_i^{(\mu^t)} \1_n \la 1 \ra) = 1$  fix a spanning map $f_{\l,\mu}$. A complex as in (\ref{eq:cpx}) will have differentials of the form $a_{\l,\mu} f_{\l,\mu}$ for some scalars $a_{\l,\mu} \in \k$. We need to show that 
\begin{enumerate}
\item if the complex in (\ref{eq:cpx}) is indecomposable then $a_{\l,\mu} \ne 0$ 
\item if $\{ a_{\l,\mu} \}$ and $\{ a'_{\l,\mu} \}$ are two choices of scalars then they are equivalent via some homotopy.
\end{enumerate} 
{\bf Proof of (1).} We proceed by (decreasing) induction on the cohomological degree. The base case is the map $a_{1,1} f_{1,1}: \P_i \Q_i \1_n \la -1 \ra \rightarrow \1_n$. Clearly $a_{1,1} \ne 0$ because otherwise $\T_i \1_n$ would decompose. Now suppose $a_{\beta,\alpha} \ne 0$ for all $\beta$ with $|\beta| > \ell$ (for some $\ell$).

{\bf Claim.} If $a_{\l,\mu} = 0$ for some $\l$ with $|\l| = \ell$ then all differentials out of $\P_i^{(\l)} \Q_i^{(\l^t)} \1_n$ are zero. 

To see this consider another differential $a_{\l,\mu'} f_{\l,\mu'}$. These two maps are part of a unique skew-commutative square 
\begin{equation}\label{eq:commsquare}
\xymatrix{
& \P_i^{(\mu)} \Q_i^{(\mu^t)} \1_{n} \la 1 \ra \ar[dr]^{a_{\mu,\nu} f_{\mu,\nu}} & \\
\P_i^{(\l)} \Q_i^{(\l^t)} \1_n \ar[ru]^{a_{\l,\mu} f_{\l,\mu} = 0} \ar[rd]^{a_{\l,\mu'} f_{\l,\mu'}} &  & \P_i^{(\nu)} \Q_i^{(\nu^t)} \1_n  \la 2 \ra \\
& \P_i^{(\mu')} \Q_i^{({\mu'}^t)} \1_n \la 1 \ra \ar[ru]^{a_{\mu',\nu} f_{\mu',\nu}} & 
}
\end{equation}
By induction the right two maps are nonzero, which means that $a_{\mu',\nu}=0$ since $f_{\l',\nu} f_{\l,\mu'} \ne 0$ by Lemma \ref{lem:3}. This proves our claim. 

Now, since $\T_i$ is indecomposable and all maps out of $\P_i^{(\l)} \Q_i^{(\l^t)} \1_n$ are zero there must be a sequence of nonzero differentials which connect $\P_i^{(\l)} \Q_i^{(\l^t)} \1_n$ to some other $\P_i^{(\mu)} \Q_i^{({\mu}^t)} \1_n$. Such a path is depicted by the solid arrows in the figure below. 

$$
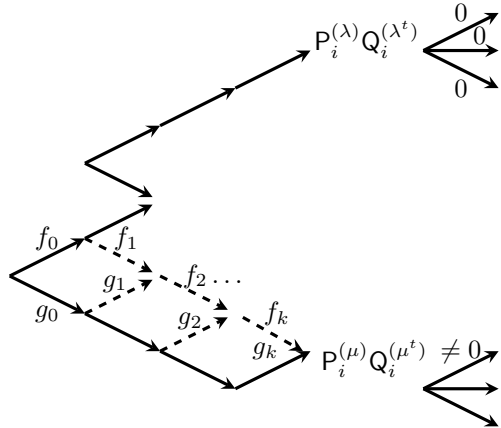
\begin{tikzpicture}[>=stealth]
\draw (0,0) -- (1,.5)[->] [very thick];
\draw (1,.5) -- (2,1)[->] [very thick];
\draw (2,1) -- (3,1.5)[->] [very thick];
\draw (3.75,1.65) node {$\P_i^{(\l)}\Q_i^{(\l^t)}$};
\draw (4.5,1.5) -- (5.5,2)[->] [very thick];
\draw (4.5,1.5) -- (5.5,1.5)[->] [very thick];
\draw (4.5,1.5) -- (5.5,1)[->] [very thick];

\draw (5,2) node {$0$};
\draw (5.25,1.7) node {$0$};
\draw (5,1) node {$0$};

\draw (0,0) -- (.9,-.45)[->] [very thick];

\draw (0,-1) -- (.9,-.55)[->] [very thick];
\draw (-1,-1.5) -- (0,-1)[->] [very thick];

\draw (-1,-1.5) -- (0,-2)[->] [very thick];
\draw (0,-2) -- (1,-2.5)[->] [very thick];
\draw (1,-2.5) -- (2,-3)[->] [very thick];
\draw (2,-3) -- (3,-2.5)[->] [very thick];
\draw (3.85,-2.6) node {$\P_i^{(\mu)}\Q_i^{(\mu^t)}$};

\draw (4.5,-3) -- (5.5,-2.5)[->] [very thick];
\draw (4.5,-3) -- (5.5,-3)[->] [very thick];
\draw (4.5,-3) -- (5.5,-3.5)[->] [very thick];

\draw (5,-2.5) node {$\neq0$};

\draw (0,-1) -- (.9,-1.45)[->] [dashed,very thick];
\draw (1,-1.5) -- (1.9,-1.95)[->] [dashed,very thick];
\draw (2.1,-2.05) -- (2.9,-2.5)[->] [dashed,very thick];

\draw (0,-2) -- (.9,-1.55)[->] [dashed,very thick];
\draw (1,-2.5) -- (1.9,-2.05)[->] [dashed,very thick];

\draw (.4,-1.6) node {$g_1$};
\draw (1.4,-2.1) node {$g_2$};
\draw (2.4,-2.5) node {$g_k$};

\draw (-.5,-1) node {$f_0$};
\draw (-.5,-2) node {$g_0$};

\draw (.55,-1) node {$f_1$};
\draw (1.7,-1.5) node {$f_2\hdots$};
\draw (2.55,-2) node {$f_k$};
\end{tikzpicture}
$$
By assumption, in the path connecting $\P_i^{(\l)} \Q_i^{(\l^t)} \1_n$ and $\P_i^{(\mu)} \Q_i^{(\mu^t)} \1_n$, the outward differentials (and in particular the map $g_k$ at the end) are assumed to be nonzero.  Without loss of generality we may assume that the path is minimal with respect to the area to its right. Now consider the square formed by $f_0,g_0,f_1,g_1$ (note that $f_1,g_1$ are uniquely determined by $f_0,g_0$). The skew-commutativity of this square means that $f_1$ and $g_1$ are either both zero or both nonzero. If they were both nonzero then one would obtain a smaller path which uses $f_1,g_1$ in place of $f_0,g_0$. So we must have $f_1=0=g_1$. 

Now, the skew-commutativity of the subsequent squares means that $g_1=0 \Rightarrow g_2=0 \Rightarrow \dots \Rightarrow g_k = 0$ which is a contradiction. Thus $a_{\l,\mu} \ne 0$ and we are done.

{\bf Proof of (2).} By (1) we know that $a_{1,1} \ne 0 \ne a'_{1,1}$. So we can rescale $\P_i \Q_i \1_n$ (the second term from the right in $\T_i \1_n$) so that $a_{1,1} = a'_{1,1}$. Now suppose by induction that $a_{\alpha,\beta} = a'_{\alpha,\beta}$ for all $|\beta| > \ell$. For any $\l$ with $|\l| = \ell$ and differential $\P_i^{(\l)} \Q_i^{(\l^t)} \1_n \rightarrow \P_i^{(\mu)} \Q_i^{(\mu^t)} \1_n$ we can rescale $\P_i^{(\l)} \Q_i^{(\l^t)}$ so that $a_{\l,\mu} = a'_{\l,\mu}$. 

{\bf Claim.} For any other differential out of $\P_i^{(\l)} \Q_i^{(\l^t)}$ we have $a_{\l,\mu'} = a'_{\l,\mu'}$. To see this consider again the square in (\ref{eq:commsquare}). By induction we know that $a_{\mu,\nu} = a'_{\mu,\nu}$ and $a_{\mu',\nu} = a'_{\mu',\nu}$. Then by the rescaling above we also have $a_{\l,\mu} = a'_{\l,\mu}$ so the skew-commutativity of the square also gives $a_{\l,\mu'} = a'_{\l,\mu'}$. 

Thus we can continue this way and scale things so that $a_{\l,\mu} = a'_{\l,\mu}$ for all $\mu \subset \l$ (which proves (2)). 
\end{proof}

\begin{Remark}\label{rem:1} 
To define the differentials in (\ref{eq:cpx}) one can choose arbitrary nonzero maps $a_{\l,\mu} f_{\l,\mu}: \P_i^{(\l)} \Q_i^{(\l^t)} \1_n \rightarrow \P_i^{(\mu)} \Q_i^{(\mu^t)} \la 1 \ra \1_n$ and then rescale them as in the proof above to obtain a complex ({\it i.e.} so the square in (\ref{eq:commsquare}) is skew-commutative). 
\end{Remark} 

\section{The braid relations}\label{sec:braidrels}

\subsection{The braid relations in $\K_{Fock}$}\label{sec:braidfock}

We first work with the 2-category $\Kom(K_{Fock})$. To simplify notation we will denote $B := B^D_\ep(1)$ and $B_i := Be_{i,1}$ and ${}_iB := e_{i,1} B$ the natural $(B,\k)$ and $(\k,B)$ bimodules.  Define the complex $\Sigma_i \1_1$ of $(B,B)$-bimodules as
$$\Sigma_i \1_1 := B_i \otimes_\k {}_iB \la -2 \ra \rightarrow B$$
where the map is multiplication (or equivalently given by a cap) and the right hand term $B$ is in cohomological degree zero. Similarly, we define 
$$\Sigma_i^{-1} \1_1 := B \rightarrow B_i \otimes_\k {}_iB \la 2 \ra $$
where the map is given by a cap and the left hand term $B$ is in cohomological degree zero. 

\begin{prop}\label{prop:braids}
The complexes $\Sigma_i \1_1$ satisfy the braid relations of $\Br(D)$ where the inverse of $\Sigma_i \1_1$ is the complex $\Sigma_i^{-1} \1_1$. In other words, in the homotopy category of $(B,B)$-bimodules we have the following homotopy equivalences 
\begin{itemize}
\item $\Sigma_i \Sigma_i^{-1} \1_1 \xrightarrow{\sim} \1_1$ and $\Sigma_i^{-1} \Sigma_i \1_1 \xrightarrow{\sim} \1_1$,
\item $\Sigma_i \Sigma_j \Sigma_i \1_1 \xrightarrow{\sim} \Sigma_j \Sigma_i \Sigma_j \1_1$ if $\la i,j \ra = -1$, 
\item $\Sigma_i \Sigma_j \1_1 \xrightarrow{\sim} \Sigma_j \Sigma_i \1_1$ if $\la i,j \ra = 0$. 
\end{itemize}
\end{prop}
\begin{proof}
These relations are proven in \cite{HK}, following earlier work \cite{KS} in type $A$.
\end{proof}

\begin{cor}\label{cor:braids}
In the homotopy category $\Kom(\K_{Fock})$ we have the following homotopy equivalences 
\begin{itemize}
\item $\Sigma_i^{[n]} ({\Sigma_i^{-1}})^{[n]} \1_n \xrightarrow{\sim} \1_n$ and $({\Sigma_i^{-1}})^{[n]} \Sigma_i^{[n]} \1_n \xrightarrow{\sim} \1_n$,
\item $\Sigma_i^{[n]} \Sigma_j^{[n]} \Sigma_i^{[n]} \1_n \xrightarrow{\sim} \Sigma_j^{[n]} \Sigma_i^{[n]} \Sigma_j^{[n]} \1_n$ if $\la i,j \ra = -1$,
\item $\Sigma_i^{[n]} \Sigma_j^{[n]} \1_n \xrightarrow{\sim} \Sigma_j^{[n]} \Sigma_i^{[n]} \1_n$ if $\la i,j \ra = 0$. 
\end{itemize}
\end{cor}
\begin{proof}
This is an immediate consequence of the naturality of the wreath functor $(\cdot) \mapsto (\cdot)^{[n]}$ as discussed in section \ref{sec:wreath}. For example, the homotopy equivalence $\Sigma_1^{-1} \Sigma_1 \1_1 \xrightarrow{\sim} \1_1$ induces 
$$\Sigma_i^{[n]} (\Sigma_i^{-1})^{[n]} \1_n \cong (\Sigma_i \Sigma_i^{-1} \1_1)^{[n]} \xrightarrow{\sim} (\1_1)^{[n]} = \1_n.$$
\end{proof}

The next theorem relates $\Sigma_i^{[n]} \1_n$ and the complex $\T_i \1_n$ defined in (\ref{eq:cpx}).  

\begin{Theorem}\label{thm:braids}
The complex $\Sigma_i^{[n]}$ is isomorphic to the complex of $(B^D_\ep(n),B^D_\ep(n))$-bimodules
\begin{equation}\label{eq:cpx2}
\bigoplus_{\l \vdash n} \P_i^{(\l)} \Q_i^{(\l^t)} \la -n \ra \1_n \rightarrow \dots \rightarrow \bigoplus_{\l \vdash d} \P_i^{(\l)} \Q_i^{(\l^t)} \la -d \ra \1_n \rightarrow \dots \rightarrow \P_i \Q_i \la -1 \ra \1_n\rightarrow \1_n
\end{equation}
defined in section \ref{sec:cpx}. 
\end{Theorem}

\subsection{Proof of Theorem \ref{thm:braids}}

The statement of the Theorem \ref{thm:braids} only involves one vertex of the Dynkin diagram $D$. To simplify notation we will assume that $D$ contains only one vertex so that $B_\ep^D \cong \k[t]/t^2$. Notice that $\deg t = 2$ so $B_\ep^D$ is commutative (not supercommutative). We will use the notation $B := B_\ep^D$ and $B_n := B_\ep^D(n) = B^{\otimes n} \rtimes \k[S_n]$. Thus $\Sigma \1_1 = B' \rightarrow B$ as a $(B,B)$-bimodule where we denote $B' := B \otimes_\k B$. 

By definition the term of the complex $\Sigma_i^{[n]}$ in cohomological degree $-d$ is $\bigoplus_{\ul} B_\ul \rtimes \k[S_n]$ where $\ul = (\ell_1, \dots, \ell_d)$ with $0 \le \ell_1 < \dots < \ell_d \le n$ and $B_\ul$ is a tensor product over $\k$ of $B$s and $B'$s where the $B'$s occur exactly in positions $\ell_1, \dots, \ell_d$. For example, 
$$B_{(2,3)} = B \otimes_\k B' \otimes_\k B' \otimes_\k B \text{ where } n=4, d=2.$$

\begin{Remark}
The tensor product above is that of supervector spaces where $\deg B = 0$ and $\deg B' = 1$. So, for instance, the involution $(23) \in S_4$ acts on $B_{(2,3)}$ by 
$$(23) \cdot (b_1 \otimes b_2 \otimes b_3 \otimes b_4) \mapsto - (b_1 \otimes b_3 \otimes b_2 \otimes b_4)$$
while $(12) \in S_4$ acts as a map $B_{(2,3)} \rightarrow B_{(1,3)}$ by 
$$(12) \cdot (b_1 \otimes b_2 \otimes b_3 \otimes b_4) \mapsto (b_2 \otimes b_1 \otimes b_3 \otimes b_4).$$
\end{Remark}

\begin{prop}\label{prop:1}
There is an isomorphism of $(B_n,B_n)$-bimodules 
$$\bigoplus_{\l \vdash d} B_n e_\l \otimes_{B_{n-d}} e_{\l^t} B_n \xrightarrow{\sim} \bigoplus_{\ul} B_\ul \rtimes \k[S_n]$$
where $B_{n-d}$ acts on $B_n$ via the embedding 
$$B^{\otimes n-d} \ni b_1 \otimes \dots \otimes b_{n-d} \mapsto 1 \otimes \dots \otimes 1 \otimes b_1 \otimes \dots \otimes b_{n-d} \in B^{\otimes n}$$ 
and $e_\l \in \k[S_d] \subset B_d \subset B_n$ acts on the first $d$ factors (so it commutes with $B_{n-d}$). 
\end{prop}
\begin{proof}
Consider the map $\phi: B_n \otimes_{B_{n-d}} B_n \rightarrow \oplus_{\ul} B_\ul \rtimes \k[S_n]$ given by 
$$(1^{\otimes n},1) \otimes (1^{\otimes n},1) \mapsto (1^{\otimes d} \otimes 1^{\otimes n-d},1) \in B_{(1,2,\dots,d)} = (B')^{\otimes d} \otimes B^{\otimes n-d} \subset B_{\ul}.$$
This map is well defined since if $(b_1 \otimes \dots \otimes b_{n-d},\sigma) \in B_{n-d}$ then 
\begin{align*}
\phi((1^{\otimes d} \otimes b_1 \otimes \dots \otimes b_{n-d},\sigma) \otimes (1^{\otimes n},1))
& = (1^{\otimes d} \otimes b_1 \otimes \dots \otimes b_{n-d},\sigma) \\
& = \phi((1^{\otimes n},1) \otimes (1^{\otimes d} \otimes b_1 \otimes \dots \otimes b_{n-d},\sigma)).
\end{align*}
Moreover, since $(1^{\otimes d} \otimes 1^{\otimes n-d},1) \in B_{(1,2,\dots,d)}$ generates $B_{\ul}$ as a $(B_n,B_n)$-bimodule $\phi$ is also surjective. 

Now, 
$$B_n \otimes_{B_{n-d}} B_n \cong \bigoplus_{\l,\l' \vdash d} \left( B_n e_\l \otimes_{B_{n-d}} e_{\l'} B_n \right)^{\oplus d_\l d_{\l'}}$$
where $d_\l$ (resp. $d_{\l'}$) is the dimension of the irreducible $\k[S_d]$-module $V_\l$ (resp. $V_{\l'}$) indexed by the partition $\l$ (resp. $\l'$). Now, for $s_i \in S_d$ we have
$$\phi((1^{\otimes n}, s_i) \otimes (1^{\otimes n}, 1)) = - (1^{\otimes n}, s_i) = (1^{\otimes n}, \tau(s_i))$$
where the minus sign is because $s_i$ acts on $B' \otimes_\k B'$ by $s_i \cdot (1 \otimes 1) = - (1 \otimes 1)$. Subsequently, 
$$\phi((1^{\otimes n}, e_{\l}) \otimes (1^{\otimes n}, e_{\l'})) = (1^{\otimes n}, \tau(e_{\l}) e_{\l'}) = \delta_{\l^t,\l'} (1^{\otimes n}, e_{\l^t}).$$
This means that $B_n e_\l \otimes_{B_{n-d}} e_{\l'} B_n$ is in the kernel of $\phi$ unless $\l' = \l^t$. We conclude that
\begin{equation}\label{eq:iso}
\phi: \bigoplus_{\l \vdash d} B_n e_{\l} \otimes_{B_{n-d}} e_{\l^t} B_n \longrightarrow \oplus_{\ul} B_{\ul} \rtimes \k[S_n]
\end{equation}
is surjective. 

To show that the map in (\ref{eq:iso}) is an isomorphism we compute the dimensions over $\k$ of both sides. On the one hand $\dim_\k B_\ul = (\dim_\k B')^d (\dim_\k B)^{n-d} = 4^d 2^{n-d}$. Thus 
\begin{equation}\label{eq:1}
\dim_\k \left( \oplus_\ul B_\ul \rtimes \k[S_n] \right) = \binom{n}{d} 2^{n+d} n!.
\end{equation}
On the other hand, 
$$\dim_\k (B_n e_\l) = \dim_\k B_n \cdot \frac{ \dim_\k V_\l}{d!} = \frac{ 2^n n! \dim_\k V_\l}{d!}$$
and likewise $\dim_\k (e_{\l^t} B_n) = \frac{2^n n! \dim_\k V_\l}{d!}$. Since $B_n e_\l$ and $e_{\l^t} B_n$ are free $B_{n-d}$ modules it follows that 
$$\dim_{\k} \left( B_n e_{\l} \otimes_{B_{n-d}} e_{\l^t} B_n \right) = \frac{1}{\dim_\k B_{n-d}} \frac{2^{2n} n! n! (\dim_\k V_\l)^2}{d!d!} = \binom{n}{d} 2^{n+d} n! \frac{(\dim_\k V_\l)^2}{d!}.$$
Summing over all $\l \vdash d$ and using that $\sum_{\l \vdash d} (\dim_\k V_\l)^2 = d!$ we get the same dimension as in (\ref{eq:1}). Thus (\ref{eq:iso}) must be an isomorphism. 
\end{proof}

Theorem \ref{thm:braids} now follows by combining Proposition \ref{prop:1} with Proposition \ref{prop:Tcpx} which says that any indecomposable complex such as that in (\ref{eq:cpx2}) is unique up to homotopy. 

\begin{cor}\label{cor:braids2}
In $\Kom(\K_{Fock})$ the complexes from (\ref{eq:cpx2}) satisfy the braid relations of $\Br(D)$. 
\end{cor}

\subsection{The braid relations in integrable 2-representations}

We now consider an arbitrary integrable 2-representation $\K$ of $\h$ and prove Theorem \ref{thm:main1} using Theorem \ref{thm:braids}. We will show that $\T_i \T_i^{-1} \1_n \xrightarrow{\sim} \1_n$ in $\Kom(\oH)$ (the proof of the other braid relations is similar). 

The composition $\T_i \T_i^{-1} \1_n$ is a (finite) complex of 1-morphisms in $\oH$. Decompose each term in this complex into indecomposables of the form $\P_i^{(\l)} \Q_i^{(\mu)} \la \ell \ra \1_n$ where $\ell \in \Z$ and $\l,\mu$ are partitions. Since $\End^k(\P_i^{(\l)} \Q_i^{(\mu)} \la \ell \ra \1_n)$ is zero if $k < 0$ and one-dimensional if $k=0$ we can restrict $\T_i \T_i^{-1} \1_n$ to these terms to get a complex of the form $\P_i^{(\l)} \Q_i^{(\mu)} \la \ell \ra \1_n \otimes_\k V_\bullet$ where $V_\bullet$ is a complex of vector spaces. 

Using the Cancellation Lemma \ref{lem:cancel} it suffices to show that $V_\bullet$ is exact (unless $\l = \mu = \emptyset$ and $\ell = 0$ in which case $V_\bullet$ should have one-dimensional cohomology in degree zero). To see this consider the image of $\T_i \T_i^{-1} \1_n$ in $\Kom(\K_{Fock})$. By Theorem \ref{thm:braids} this complex is homotopic to $\1_n$. Thus the image of $V_\bullet$ is exact and so $V_\bullet$ must also be exact. 

\begin{lemma}\label{lem:cancel}
Let $ X, Y, Z, W, U, V$ be six objects in an additive category and consider a complex
\begin{equation}\label{eq:6.5}
\dots \rightarrow U \xrightarrow{u} X \oplus Y \xrightarrow{f} Z \oplus W \xrightarrow{v} V \rightarrow \dots
\end{equation}
where $f = \left( \begin{matrix} A & B \\ C & D \end{matrix} \right)$ and $u,v$ are arbitrary morphisms. If $D: Y \rightarrow W$ is an isomorphism, then (\ref{eq:6.5}) is homotopic to a complex
\begin{eqnarray}\label{eq:new}
\dots \rightarrow U \xrightarrow{u} X \xrightarrow{A-BD^{-1}C} Z \xrightarrow{v|_Z} V \rightarrow \dots
\end{eqnarray}
\end{lemma}
\begin{proof}
The following result is a slight generalization of a lemma which Bar-Natan \cite{BN} calls ``Gaussian elimination''. For a proof see Lemma 6.2 of \cite{CL2}.
\end{proof}

\section{Convolution in triangulated 2-representations and Hilbert schemes}

Consider an affine Dynkin diagram where $\Gamma \subset SL_2(\C)$ is the finite subgroup associated to it via the McKay correspondence. Let $X_\Gamma = \widehat{\C^2/\Gamma}$ be the minimal resolution and $X^{[n]}_{\Gamma}$ the Hilbert scheme of $n$ points. 

In \cite{CL1} we constructed a Heisenberg 2-represention on the derived categories of coherent sheaves $\oplus_{n \ge 0} D_{\C^\times}(X_\Gamma^{[n]})$. Subsequently, Theorem \ref{thm:main1} induces an action of the affine braid group on $\Kom(D_{\C^\times}(X_\Gamma^{[n]}))$. We would like to explain now how this action descends to one on $D_{\C^\times}(X_\Gamma^{[n]})$. The main result we prove is the following. 

\begin{Theorem}\label{thm:geom}
For each $n \geq 0$, the complexes $\Sigma_i \1_n$ have a unique convolution 
$$\Conv(\Sigma_i \1_n) \in D_{\C^\times \times \C^\times}(X_\Gamma^{[n]} \times X_\Gamma^{[n]}).$$ 
These convolutions define an action of the affine braid group on $D_{\C^*}(X_\Gamma^{[n]})$.
\end{Theorem}

The existence of such an affine braid group action is known by combining \cite{ST,P}. The convolutions above give another interpretation of this action by showing that it arises from a categorical Heisenberg action. 

\subsection{Convolutions}

Let $A_\bullet = A_0 \xrightarrow{f_1} A_1 \rightarrow \cdots \xrightarrow{f_n} A_n $ be a sequence of objects and morphisms in a triangulated category $ \mathcal{D}$ such that $ f_{i+1} \circ f_i = 0 $. Such a sequence is called a complex.

A \textit{(right) convolution} of a complex $A_\bullet $ is any object $B$ such that there exist
\begin{enumerate}
\item objects $ A_0 = B_0, B_1, \dots, B_{n-1}, B_n = B $ and
\item morphisms $g_j : B_{j-1} \rightarrow A_j$, $h_j : A_j \rightarrow B_j $ (with $h_0 = id$)
\end{enumerate}
such that $B_{j-1} \xrightarrow{g_j} A_i \xrightarrow{h_j} B_j$ is a distinguished triangle for each $i$ and $ g_j \circ h_{j-1} = f_j $. Such a collection of data is called a \textit{Postnikov system}. We will denote $B_n$ by $\Conv(A_\bullet)$. When $n=1$ then $B_n$ is isomorphic to the usual cone. 

\begin{prop}\cite[Prop. 8.3]{CK1}\label{th:uniquecone}
Consider a complex $A_\bullet$.
\begin{enumerate}
\item
If $\Hom(A_j[k], A_{j+k+1}) = 0 $ for all $ j \ge 0, k \ge 1 $, then any two convolutions of $ (A_\bullet, f_\bullet) $ are isomorphic.
\item
If $ \Hom(A_j[k], A_{j+k+2}) = 0 $ for all $ j \ge 0, k \ge 1 $, then $ (A_\bullet, f_\bullet) $ has a convolution.
\end{enumerate}
\end{prop}

\begin{prop}\label{prop:complexzero}
Suppose $\D$ is a triangulated category which satisfies the Krull-Schmidt property. If $A_\bullet$ is a complex of objects in $\D$ which is homotopic to zero then $\Conv(A_\bullet) \cong 0$.
\end{prop}
\begin{proof}
The proof is by induction on the length of the complex. The base case is trivial. 

Now, if $A_\bullet = A_0 \xrightarrow{f_1} A_1 \xrightarrow{} B_\bullet$ is homotopic to zero then there exists a map $g_1: A_1 \rightarrow A_0$ such that $g_1 \circ f_1 = \id_{A_0}$. Since $\D$ is Krull-Schmidt this means that $A_1 \cong A_0 \oplus A_1'$ and we can rewrite $A_\bullet$ as 
\begin{equation*}
A_0 \xrightarrow{(f_1',f_1'')} A_0 \oplus A_1' \xrightarrow{h} B_\bullet
\end{equation*}
where $f_1'$ is an isomorphism. Now, if we take the first cone we obtain a commutative diagram 
\begin{eqnarray*}
\xymatrix{
A_0 \ar[d] \ar[rr]^{(f_1',f_1'')} & &  \ar[d]^{(\alpha,\beta)} A_0 \oplus A_1' \ar[rr]^{h} & & B_\bullet \ar[d]^\id \\
0 \ar[rr] & & \Cone(f_1',f_1'') \ar[rr]^u & & B_\bullet
}
\end{eqnarray*}
where $u$ is some map. Note that $\Conv(A_\bullet) \cong \Conv(\Cone(f_1',f_1'') \xrightarrow{u} B_\bullet)$ so, by induction, it suffices to show that $\Cone(f_1',f_1'') \xrightarrow{u} B_\bullet$ is homotopic to zero. 

Since $f_1'$ is an isomorphism this means $\beta$ is an isomorphism. Thus, precomposing with $\beta$ we get 
\begin{equation}\label{eq:5}
[\Cone(f_1',f_1'') \xrightarrow{u} B_\bullet] \cong [A_1' \xrightarrow{u \circ \beta} B_\bullet].
\end{equation}
Note that by commutativity of the square $u \circ \beta = h$. 

On the other hand, using the cancellation lemma \ref{lem:cancel}
$$[A_0 \xrightarrow{(f_1',f_1'')} A_0 \oplus A_1' \xrightarrow{h} B_\bullet] \cong [0 \rightarrow A_1' \xrightarrow{h} B_\bullet].$$
Thus, combining this with (\ref{eq:5}) we get that $\Cone(f_1',f_1'') \xrightarrow{u} B_\bullet$ is homotopic to zero. 
\end{proof}

\subsection{Proof of Theorem \ref{thm:geom}}

First we show that $\Sigma_i \1_n$ has a unique convolution. By Proposition \ref{th:uniquecone} it suffices to check that for any partitions $\l,\mu,\nu$ with $|\l|-|\mu| \ge 2$ and $|\l|-|\nu| \ge 3$ we have
\begin{equation}\label{eq:6}
\Hom(\P_i^{(\l)} \Q_i^{(\l^t)}, \P_i^{(\mu)} \Q_i^{(\mu^t)} \la 1 \ra) = 0 = \Hom(\P_i^{(\l)} \Q_i^{(\l^t)}, \P_i^{(\nu)} \Q_i^{(\nu^t)} \la 2 \ra)
\end{equation}
where $\la 1 \ra = [1]$. In the 2-category $\H$ this is clear because any map in the first (resp. second) $\Hom$ space above, must contain two (resp. three) adjunction maps and must therefore have degree at least two (resp. three). On the other hand, this can also be seen by using adjunction and the commutation relations in $\H$ to reduce the statements above to the statement that $\End(\1_0,\1_0 \la s \ra) = 0$ for $s < 0$. Since this is also true in our 2-representation consisting of coherent sheaves on $X_\Gamma^{[n]}$ the vanishing of (\ref{eq:6}) holds there too. 

It remains to show that these unique convolutions $\Conv(\Sigma_i \1_n)$ satisfy the affine braid relations. This follows from the affine braid relations satisfied in $\Kom(\H)$. We illustrate this by proving that
$$\Conv(\Sigma_i \1_n) \circ \Conv(\Sigma_i^{-1} \1_n) \cong \1_n.$$
First, we have that $\Sigma_i \1_n \circ \Sigma_i^{-1} \1_n \cong \1_n$ in the homotopy category. This means that we have a map $\Sigma_i \1_n \circ \Sigma_i^{-1} \1_n \rightarrow \1_n$ whose cone is homotopic to zero. By Proposition \ref{prop:complexzero} we know that any convolution of this cone is zero and hence there is an isomorphism 
$$\Conv(\Sigma_i \1_n \circ \Sigma_i^{-1} \1_n) \xrightarrow{\sim} \1_n$$
for any choice of convolution (we do not know that it has a unique convolution). On the other hand, $\Conv(\Sigma_i \1_n) \circ \Conv(\Sigma_i^{-1} \1_n)$ is {\it some} convolution of $\Sigma_i \1_n \circ \Sigma_i^{-1} \1_n$. Thus 
$$\Conv(\Sigma_i \1_n) \circ \Conv(\Sigma_i^{-1} \1_n) \cong \Conv(\Sigma_i \1_n \circ \Sigma_i^{-1} \1_n) \cong \1_n$$ 
and we are done. 

\begin{Remark}\label{rem:conv}
Although Theorem \ref{thm:geom} involves the triangulated category of coherent sheaves on a surface, the same proof works to show that the conclusions in that theorem hold for any (graded) triangulated category where 
$$\Hom(\1_n,\1_n \la \ell \ra) = 
\begin{cases} 
0 & \text{ if } \ell < 0 \\
\C & \text{ if } \ell = 0 \\
\text{finite dimensional } & \text{ if } \ell > 0.
\end{cases}$$
\end{Remark}

\subsection{A larger group}

As mentioned in the introduction, $D_{\C^\times}(X_\Gamma)$ carries an action of the affine braid group defined using Seidel-Thomas twists \cite{ST}. The affine braid group action constructed above coincides with its lift from $D_{\C^\times}(X_\Gamma)$ to $D_{\C^\times}(X_\Gamma^{[n]})$ using the results in \cite{P}. 

On the other hand, there is actually a larger group acting on $D_{\C^\times}(X_\Gamma^{[n]})$. It is generated by certain complexes very similar to our $\Sigma_i \1_n$. These complexes are briefly discussed in section \ref{sec:vertexops} (equation (\ref{eq:eg2}) is an example of such a complex). However, the autoequivalence such a complex generates is not the lift of any automorphism of $D_{\C^\times}(X_\Gamma)$.

\section{A braid group action on $\Kom(\H)$}\label{sec:braidH}
In this section we define an abstract action of $ \Br(D) $ on $\Kom(\H')$ where $ \H'$ is the full subcategory of $ \H$ generated by the $ \P_i $.  The quotient of $\H' $ by the ideal generated by $\1_n$ for $n < 0$ may be thought of as another categorification of Fock space.  

Recall that the 2-category $\H$ contains generating 2-morphisms $X_j^k \colon \P_j \rightarrow \P_k \la 1 \ra$ for $ \la j,k \ra=-1$  and $ T_{jk} \colon \P_j \P_k \rightarrow \P_k \P_j $ for any $j,k \in I$. In \cite{CL1} we encode these maps diagrammatically as a dot and a crossing respectively. To define a braid group action on $\H'$ we need to explain how the generators $\s_i^{\pm 1}$ act on 1-morphisms $\P_j$ and $\Q_j$ as well as on 2-morphisms $X_j^k$ and $T_{jk}$. We will then extend this action monoidally -- for example, $\s_i(\P_j\P_j) = \s_i(\P_j) \s_i(\P_j)$. 

In the appendix we will prove the following result. 

\begin{Theorem}\label{thm:main2}
The actions of $\s_i^{\pm 1}$ defined in sections \ref{sec:action1}--\ref{sec:action5} induce 2-endofunctors of the 2-category $\Kom(\H')$ which satisfy the braid relations of $\Br(D)$. 
\end{Theorem}

One can extend the results of this section to describe an action of $\Br(D)$ on the entire category $\Kom(\H)$, rather than on just the upper half $\Kom(\H')$. However, this would require checking even more relations and our interest in this paper is to understand braid group actions arising from integrable 2-representations of $\H$. For this purpose, Theorem \ref{thm:main2} is sufficient.

\subsection{The action of $\s_i^{\pm 1}$ on 1-morphisms}\label{sec:action1}

As usual, we will use $[k]$ to denote a cohomological shift to the left by $k$ and $\la k \ra$ to denote the internal grading shift of $\H$. 

We define 
$$\s_i (\P_j) := 
\begin{cases} 
\P_i \la -2 \ra \oplus \P_i \xrightarrow{(X_i^{i} \ \ 1_i)} \P_i & \text{ if } i=j \in I \\
\P_i \la -1 \ra \xrightarrow{X_i^j} \P_j & \text{ if } \la i,j \ra = -1 \\  
\P_j & \text{ if } \la i,j \ra = 0
\end{cases}$$
where the right hand terms are all in homological degree zero. Likewise, we define
$$
\s_i^{-1}(\P_j) := 
\begin{cases}
\P_i \xrightarrow{(1_i \ \ X_i^{i})} \P_i \oplus \P_i \la 2 \ra & \text{ if } i=j \in I \\
\P_j \xrightarrow{X_j^i} \P_i \la 1 \ra & \text{ if } \la i,j \ra = -1 \\
\P_j & \text{ if } \la i,j \ra = 0
\end{cases}$$

where this time the left hand terms are all in degree zero.

\subsection{The action of $\s_i$ on $X$'s}\label{sec:action2}

Suppose $\la i,j \ra = -1$. We define 
\begin{equation*}
\xymatrix{
\s_i(\P_i \la -1 \ra) \ar[d]_{\s_i(X_i^{j})} & = & \P_i \la -3 \ra \oplus \P_i \la -1 \ra \ar[rr]^{\hspace{.2in} (X_i^i \hspace{.2cm} 1) } \ar[d]_{(0 \hspace{.2cm} 1_i)} & & \P_i \la -1 \ra \ar[d]^{X_i^j}  \\
\s_i(\P_{j}) & = & \P_i \la -1 \ra \ar[rr]^{X_i^{j}} & & \P_{j} 
}
\end{equation*}
\begin{equation*}
\xymatrix{
\s_i(\P_{j} \la -1 \ra) \ar[d]_{\s_i(X_{j}^i)} & = & \P_i \la -2  \ra \ar[rr]^{X_i^{j}} \ar[d]_{(\epsilon_{ij} 1_i \hspace{.2cm} 0)} & & \P_{j} \la -1 \ra \ar[d]^{X_j^i} \\
 \s_i(\P_i) & = & \P_i \la -2 \ra \oplus \P_i  \ar[rr]^{\hspace{.2in} ( X_i^i   1_i)} & & \P_i }
\end{equation*}
where the rightmost columns are both in cohomological degree zero.

Next, suppose further that $j \ne k$ and $\la j,k \ra = -1$, $\la i,k \ra = 0$. Then define 
\begin{equation*}
\xymatrix{
\s_i(\P_{j} \la -1 \ra) \ar[d]_{\s_i(X_{j}^{k})} & = & \P_i \la -2 \ra \ar[r]^{X_i^{j}} \ar[d]^{}& \P_{j} \la -1 \ra \ar[d]^{X_{j}^{k}}   \\
\s_i(\P_{k}) & = & 0 \ar[r]^{} & \P_{k} 
}
\end{equation*}
\begin{equation*}
\xymatrix{
\s_i(\P_{k}) \ar[d]_{ \s_i(X_{k}^{j})} & = & 0 \ar[r]^{} \ar[d]^{} & \P_{k} \la -1 \ra \ar[d]^{X_{k}^{j}}   \\
\s_i(\P_{j}) & = & \P_i \la -1 \ra \ar[r]^{X_i^{j}} & \P_{j} 
}
\end{equation*}
where again the rightmost columns are in cohomological degree zero.

An unusual situation is when $i,j,k$ form a triangle, meaning $\la i,j \ra = \la j,k \ra = \la i,k \ra = -1$. In this case we define
\begin{equation*}
\xymatrix{
\s_i(\P_j) \ar[d]_{\s_i(X_j^k)} & = & \P_i \la -1 \ra \ar[r]^{X_i^j} \ar[d]^{0} & \P_j \ar[d]^{X_j^k} \\
\s_i(\P_k) & = & \P_i \ar[r]^{X_i^k} & \P_k \la 1 \ra
}
\end{equation*}
where the rightmost column is in cohomological degree zero. 

Finally, if $\la i,j \ra = \la i,k \ra = 0$ then $\s_i(X_j^k) = X_j^k$.

\subsection{The action of $\s_i^{-1}$ on $X$'s}\label{sec:action3}
Suppose $ \la i, j \ra = -1 $.  We define
\begin{equation*}
\xymatrix{
\s_i^{-1}(\P_i \la -1 \ra) \ar[d]^{\s_i^{-1}(X_i^j)} & = &  \P_i \la -1 \ra \ar[d]^{X_i^j} \ar[rr]^{(1_i \hspace{.2cm} X_i^i) \hspace{.6cm}}  & & \P_i \la -1 \ra \oplus \P_i \la 1 \ra \ar[d]^{(0 \hspace{.2cm} \epsilon_{ij} 1_i)} \\
\s_i^{-1}(\P_j) & = &  \P_{j} \ar[rr]^{X_j^i} & & \P_i \la 1 \ra
}
\end{equation*}

\begin{equation*}
\xymatrix{
\s_i^{-1}(\P_j \la -1 \ra) \ar[d]^{\s_i^{-1}(X_j^i)} & = &  \P_j \la -1 \ra \ar[d]^{X_j^i} \ar[rr]^{X_j^i} && \P_i \ar[d]^{(1_i \hspace{.2cm} 0)} \\
\s_i^{-1}(\P_i) & = &   \P_{i} \ar[rr]^{(1_i \hspace{.2cm} X_i^i)} && \P_i \oplus \P_i \la 2 \ra
}
\end{equation*}
where the leftmost columns are in cohomological degree zero.

In the case when $ i,j,k $ form a triangle so that $ \la i, j \ra = \la j, k \ra = \la i, k \ra = -1 $ we define
\begin{equation*}
\xymatrix{
\s_i^{-1}(\P_j \la -1 \ra) \ar[d]^{\s_i^{-1}(X_j^k)} & = &  \P_j \la -1 \ra \ar[r]^{X_j^i} \ar[d]^{X_j^k} & \P_i \ar[d]^{0} \\
\s_i^{-1}(\P_k) & = &  \P_k  \ar[r]^{X_k^i} & \P_i \la 1 \ra
}
\end{equation*}
where the leftmost column is in cohomological degree zero.

Next, suppose $ \la j, k \ra = -1, \la i, j \ra = -1, \la i, k \ra = 0 $.  Then define

\begin{equation*}
\xymatrix{
\s_i^{-1}(\P_j \la -1 \ra) \ar[d]^{\s_i^{-1}(X_j^k)} & = &   \P_{j} \la -1 \ra \ar[d]^{X_{j}^{k}} \ar[r]^{X_{j}^i} & \P_i \ar[d]\\
\s_i^{-1}(\P_k) & = &   \P_{k} \ar[r]^{} & 0
}.
\end{equation*}

\begin{equation*}
\xymatrix{
\s_i^{-1}(\P_k \la -1 \ra) \ar[d]^{\s_i^{-1}(X_k^j)} & = &  \P_{k} \la -1 \ra  \ar[d]^{X_{k}^{j}} \ar[r]^{} & 0 \ar[d]  \\
\s_i^{-1}(\P_j) & = &  \P_{j} \ar[r]^{X_{j}^i} & \P_i \la 1 \ra
}.
\end{equation*}
where the leftmost columns are in cohomological degree zero.

Finally if $ \la i, j \ra = \la i, k \ra = 0 $ then define $ \s_i^{-1}(X_j^k) = X_j^k $.

\subsection{The action of $\s_i$ on $T$'s}\label{sec:action4}

We define $ \s_i(T_{ii}) $ as the map of complexes
\begin{equation*}
\xymatrix{
\P_i \P_i \la -4,-2,-2,0 \ra \ar[r]^{A} \ar[dd]_{- \begin{pmatrix} T_{ii} & 0 & 0 & 0 \\ 0 & 0 & T_{ii} & 0 \\ 0 & T_{ii} & 0 & 0 \\ 0 & 0 & 0 & T_{ii} \end{pmatrix}}   & \P_i \P_i \la -2,0,-2,0 \ra \ar[rr]^{B} \ar[dd]^{\begin{pmatrix} 0 & 0 & T_{ii} & 0 \\ 0 & 0 & 0 & T_{ii} \\ T_{ii} & 0 & 0 & 0 \\ 0 & T_{ii} & 0 & 0 \end{pmatrix}}&&  \P_i \P_i  \ar[dd]^{T_{ii}}  \\
\\
\P_i \P_i \la -4,-2,-2,0 \ra \ar[r]^{A} & \P_i \P_i \la -2,0,-2,0 \ra \ar[rr]^{B} &&  \P_i \P_i 
}
\end{equation*}
where, for instance, $\P_i \P_i \la -4,-2,-2,0 \ra$ is shorthand for $\P_i \P_i \la -4 \ra \oplus \P_i \P_i \la -2 \ra \oplus \P_i \P_i \la -2 \ra \oplus \P_i \P_i$ and the rightmost column is in homological degree zero. The maps $A$ and $B$ are given by 
\begin{eqnarray*}
A &=&
\begin{pmatrix}
-1_i X_i^i & -1_i 1_i & 0 & 0 \\
0 & 0 & -1_i X_i^i & -1_i 1_i \\
X_i^i 1_i & 0 & 1_i 1_i & 0 \\
0 & -X_i^i 1_i & 0 & 1_i 1_i
\end{pmatrix} \\
B &=& 
\begin{pmatrix}
X_i^i 1_i & 1_i 1_i & 1_i X_i^i & 1_i 1_i
\end{pmatrix}.
\end{eqnarray*}

Now suppose $\la i,j \ra = -1$. Then we define $\s_i(T_{ij})$ by
\begin{equation*}
\xymatrix{
\s_i(\P_i \P_j) \ar[dd]^{\s_i(T_{ij})} & = & \P_i \P_i \la -3 \ra \oplus \P_i \P_i \la -1 \ra \ar[r]^{A \hspace{.3in}} \ar[dd]^{\begin{pmatrix} -T_{ii} & 0 \\ 0 & -T_{ii} \end{pmatrix}}   & \P_i \P_j \la -2 \ra \oplus \P_i \P_j \oplus \P_i \P_i \la -1 \ra  \ar[rr]^{\hspace{.5in} B} 
\ar[dd]^{\begin{pmatrix} T_{ij} & 0 & 0 \\ 0 & T_{ij} & 0 \\ 0 & 0 & T_{ij} \end{pmatrix}}&&  \P_i \P_j  \ar[dd]^{T_{ij}}  \\
\\
\s_i(\P_j \P_i) & = & \P_i \P_i \la -3 \ra \oplus \P_i \P_i \la -1 \ra \ar[r]^{C \hspace{.3in}} & \P_j \P_i \la -2 \ra \oplus \P_j \P_i \oplus \P_i \P_i \la -1 \ra \ar[rr]^{\hspace{.5in} D} &&  \P_j \P_i 
}
\end{equation*}
where the rightmost column is in cohomological degree zero and 
\begin{align*}
A &=
\begin{pmatrix}
-1_i X_i^k & 0 \\
0 & -1_i X_i^k \\
X_i^i 1_i & 1_i 1_i
\end{pmatrix} 
\ \ \
C = \begin{pmatrix}
X_i^k 1_i & 0 \\
0 & X_i^k 1_i \\
-1_i X_i^i & -1_i 1_i
\end{pmatrix} \\
B &= 
\begin{pmatrix}
X_i^i 1_k & 1_i 1_k & 1_i X_i^k
\end{pmatrix}
\ \ \ 
D = 
\begin{pmatrix}
1_k X_i^i & 1_k 1_i & X_i^k 1_i
\end{pmatrix}.
\end{align*}

Likewise, if $\la i,j \ra = -1$ then we define $\s_i(T_{ji})$ by 
\begin{equation*}
\xymatrix{
\s_i(\P_j \P_i) \ar[dd]^{\s_i(T_{ji})} & = & \P_i \P_i \la -3 \ra \oplus \P_i \P_i \la -1 \ra \ar[r]^{A \hspace{.3in}} \ar[dd]^{\begin{pmatrix} -T_{ii} & 0 \\ 0 & -T_{ii} \end{pmatrix}}   & \P_j \P_i \la -2 \ra \oplus \P_j \P_i \oplus \P_i \P_i \la -1 \ra  \ar[rr]^{\hspace{.5in} B} 
\ar[dd]^{\begin{pmatrix} T_{ji} & 0 & 0 \\ 0 & T_{ji} & 0 \\ 0 & 0 & T_{ii} \end{pmatrix}}&&  \P_j \P_i  \ar[dd]^{T_{ji}}  \\
\\
\s_i(\P_i \P_j) & = & \P_i \P_i \la -3 \ra \oplus \P_i \P_i \la -1 \ra \ar[r]^{C \hspace{.3in}} & \P_i \P_j \la -2 \ra \oplus \P_i \P_j \oplus \P_i \P_i \la -1 \ra \ar[rr]^{\hspace{.5in} D} &&  \P_i \P_j 
}
\end{equation*}
where the rightmost column is in cohomological degree zero and 
\begin{align*}
A &=
\begin{pmatrix}
X_i^j 1_i & 0 \\
0 & X_i^j 1_i \\
-1_i X_i^i & -1_i 1_i
\end{pmatrix}
\ \ \ \ \
C=
\begin{pmatrix}
-1_i X_i^j & 0 \\
0 & -1_i X_i^j \\
X_i^i 1_i & 1_i 1_i
\end{pmatrix} \\
B &= \begin{pmatrix} 1_j X_i^i & 1_j 1_i & X_i^j 1_i \end{pmatrix} \ \ \ \
D = \begin{pmatrix} X_i^i 1_j & 1_i 1_j & 1_i X_i^j \end{pmatrix}.
\end{align*}

On the other hand, if $\la i,j \ra = 0$ then $\s_i(T_{ij})$ is given by 
\begin{equation*}
\xymatrix{
\s_i(\P_i \P_j) \ar[dd]_{\s_i(T_{ij})} & = & \P_i \P_j \la -2 \ra \oplus \P_i \P_j \ar[rr]^{\hspace{.3in} \begin{pmatrix} X_i^i 1_j & 1_i 1_j \end{pmatrix}} \ar[dd]^{\begin{pmatrix} T_{ij} & 0 \\ 0 &T_{ij} \end{pmatrix}} && \P_i \P_j \ar[dd]^{T_{ij}} \\
\\
\s_i(\P_j \P_i) & = & \P_j \P_i \la -2 \ra \oplus \P_j \P_i \ar[rr]^{\hspace{.3in} \begin{pmatrix} 1_j X_i^i & 1_j 1_i \end{pmatrix}} && \P_j \P_i
}
\end{equation*}
where the rightmost column is in cohomological degree zero. 

Likewise, if $\la i,j \ra = 0$ then $\s_i(T_{ji})$ is given by
\begin{equation*}
\xymatrix{
\s_i(\P_j \P_i) \ar[dd]^{\s_i(T_{ji})} & = & \P_j \P_i \la -2 \ra \oplus \P_j \P_i \ar[rr]^{\hspace{.3in} \begin{pmatrix} 1_j X_i^i & 1_j 1_i \end{pmatrix}} \ar[dd]^{\begin{pmatrix} T_{ji} & 0 \\ 0 &T_{ji} \end{pmatrix}} && \P_j \P_i \ar[dd]^{T_{ji}}\\
\\
\s_i(\P_i \P_j) & = & \P_i \P_j \la -2 \ra \oplus \P_i \P_j \ar[rr]^{\hspace{.3in} \begin{pmatrix} X_i^i 1_j & 1_i 1_j \end{pmatrix}} && \P_i \P_j
}
\end{equation*}
where the rightmost column is in cohomological degree zero. 

Next, suppose $\la i,j \ra = -1 = \la i,k \ra$. Define $\s_i(T_{jk})$ by 
\begin{equation*}
\xymatrix{
\s_i(\P_j \P_k) \ar[dd]^{\s_i(T_{jk})} & = & \P_i \P_i \la -2 \ra \ar[rr]^{\begin{pmatrix} 1_i X_i^k \\ X_i^j 1_i \end{pmatrix} \hspace{.5in}} \ar[dd]^{T_{ii}} && \P_i \P_k \la -1 \ra \oplus \P_j \P_i \la -1 \ra \ar[rrr]^{\hspace{.5in} \begin{pmatrix} X_i^j 1_k & 1_j X_i^k \end{pmatrix}} \ar[dd]^{\begin{pmatrix} 0 & T_{ji} \\ T_{ik} & 0 \end{pmatrix}}  &&& \P_j \P_k \ar[dd]^{T_{jk}}  \\
\\
\s_i^{}(\P_j \P_k) & = &\P_i \P_i \la -2 \ra \ar[rr]^{\begin{pmatrix} 1_i X_i^j \\ X_i^k 1_i  \end{pmatrix} \hspace{.5in}} && \P_i \P_j \la -1 \ra \oplus \P_k \P_i \la -1 \ra \ar[rrr]^{\hspace{.5in} \begin{pmatrix} X_i^k 1_j  & 1_k X_i^j \end{pmatrix}} &&& \P_k \P_j
}
\end{equation*}
where the rightmost column is in cohomological degree zero. 

On the other hand, if $\la i,j \ra = -1$ and $\la i,k \ra = 0$ then $\s_i(T_{jk})$ is defined by 
\begin{equation*}
\xymatrix{
\s_i(\P_j \P_k) \ar[d]_{\s_i(T_{jk})} & = & \P_i \P_k \la -1 \ra \ar[r]^{X_i^j 1_k} \ar[d]^{T_{ik}} & \P_j \P_k \ar[d]^{T_{jk}} \\
\s_i(\P_k \P_j) & = & \P_k \P_i \la -1 \ra \ar[r]^{1_k X_i^j} & \P_k \P_j
}
\end{equation*}
while $\s_i(T_{kj})$ is defined by 
\begin{equation*}
\xymatrix{
\s_i(\P_k \P_j) \ar[d]_{\s_i(T_{kj})} & = & \P_k \P_i \la -1 \ra \ar[r]^{1_k X_i^j} \ar[d]^{T_{ki}} & \P_k \P_j \ar[d]^{T_{kj}} \\
\s_i(\P_j \P_k) & = & \P_i \P_k \la -1 \ra \ar[r]^{X_i^j 1_k} & \P_j \P_k
}
\end{equation*}
where both the rightmost columns are in cohomological degree zero.

Finally, if $\la i, j \ra = 0 = \la i, k \ra$ then $\s_i(T_{jk}) = T_{jk}$.  

\subsection{The action of $ \s_i^{-1} $ on $T$'s}\label{sec:action5}

We define $ \s_i^{-1}(T_{ii}) $ as a map of complexes:
\begin{equation*}
\xymatrix{
\P_i \P_i \ar[r]^{\hspace{-.8in} A} \ar[dd]^{T_{ii}} & \P_i \P_i \oplus \P_i \P_i \la 2 \ra \oplus \P_i \P_i \oplus \P_i \P_i \la 2 \ra \ar[rr]^{\hspace{.5in} B} \ar[dd]^{\begin{pmatrix} 0 & 0 & T_{ii} & 0 \\ 0 & 0 & 0 & T_{ii} \\ T_{ii} & 0 & 0 & 0 \\ 0 & T_{ii} & 0 & 0 \end{pmatrix}} && \P_i \P_i \la 0, 2, 2, 4 \ra \ar[dd]^{-\begin{pmatrix} T_{ii} & 0 & 0 & 0 \\ 0 & 0 & T_{ii} & 0 \\ 0 & T_{ii} & 0 & 0 \\ 0 & 0 & 0 & T_{ii} \end{pmatrix}}\\
\\
\P_i \P_i \ar[r]^{\hspace{-.8in} A} & \P_i \P_i \oplus \P_i \P_i \la 2 \ra \oplus \P_i \P_i \oplus \P_i \P_i \la 2 \ra \ar[rr]^{\hspace{.5in} B} && \P_i \P_i \la 0, 2, 2, 4 \ra
}
\end{equation*}
where the leftmost column is in cohomological degree zero.  The maps $ A $ and $ B $ are given by

$$A = \begin{pmatrix}
1_i 1_i \\
1_i X_i^i  \\
1_i 1_i \\
X_i^i 1_i
\end{pmatrix}
\ \ \ \text{ and }  \ \ \ 
B =
\begin{pmatrix}
1_i 1_i & 0 & -1_i 1_i & 0 \\
0 & 1_i 1_i & -1_i X_i^i & 0 \\
X_i^i 1_i & 0 & 0 & -1_i 1_i  \\
0 & X_i^i 1_i & 0 & -1_i X_i^i
\end{pmatrix}.$$

Now suppose $ \la i, j \ra = -1 $.  Then we
define $ \s_i^{-1}(T_{ij}) $ as a map of complexes by
\begin{equation*}
\xymatrix{
\s_i^{-1}(\P_i \P_j) \ar[dd]^{\s_i^{-1}(T_{ij})} &=&  \P_i \P_j \ar[r]^{A \hspace{.6in}} \ar[dd]^{T_{ij}}   & \P_i \P_j \oplus \P_i \P_j \la 2 \ra \oplus \P_i \P_i \la 1 \ra  \ar[rr]^{B} 
\ar[dd]^{\begin{pmatrix} T_{ij} & 0 & 0 \\ 0 & T_{ij} & 0 \\ 0 & 0 & T_{ii} \end{pmatrix}}&&    \P_i \P_i \la 1 \ra \oplus \P_i \P_i \la 3 \ra \ar[dd]^{\begin{pmatrix} -T_{ii} & 0 \\ 0 & -T_{ii} \end{pmatrix}}  \\
\\
\s_i^{-1}(\P_j \P_i) &=&  \P_j \P_i  \ar[r]^{C \hspace{.6in}} & \P_j \P_i \oplus \P_j \P_i \la 2 \ra \oplus \P_i \P_i \la 1 \ra \ar[rr]^{D} &&  \P_i \P_i \la 1 \ra \oplus \P_i \P_i \la 3 \ra
}
\end{equation*}
where the leftmost column is in cohomological degree zero and 

\begin{align*}
A &=
\begin{pmatrix}
1_i 1_k \\
X_i^i 1_k \\
1_i X_k^i
\end{pmatrix}
\ \ \
B = 
\begin{pmatrix}
-1_i X_k^i & 0 & 1_i 1_i \\
0 & -1_i X_k^i & X_i^i 1_i
\end{pmatrix}
\\
C&=
\begin{pmatrix}
1_k 1_i \\
1_k X_i^i \\
X_k^i 1_i
\end{pmatrix}
\ \ \
D = 
\begin{pmatrix}
X_k^i 1_i & 0 & -1_i 1_i \\
0 & X_k^i 1_i & -1_i X_i^i
\end{pmatrix}.
\end{align*}

Likewise, if $ \la i, j \ra = -1 $ then we define $ \s_i^{-1}(T_{ji}) $ by
\begin{equation*}
\xymatrix{
\s_i^{-1}(\P_j \P_i) \ar[dd]^{\s_i^{-1}(T_{ji})} &=&  \P_j \P_i \ar[r]^{A \hspace{.6in}} \ar[dd]^{T_{ji}} & \P_j \P_i \oplus \P_j \P_i \la 2 \ra \oplus \P_i \P_i \la 1 \ra  \ar[rr]^{B} 
\ar[dd]^{\begin{pmatrix} T_{ji} & 0 & 0 \\ 0 & T_{ji} & 0 \\ 0 & 0 & T_{ii} \end{pmatrix}}&&  \P_i \P_i \la 1 \ra \oplus \P_i \P_i \la 3 \ra  \ar[dd]^{\begin{pmatrix} -T_{ii} & 0 \\ 0 & -T_{ii} \end{pmatrix}}  \\
\\
\s_i^{-1}(\P_i \P_j) &=&  \P_i \P_j  \ar[r]^{C \hspace{.6in}} & \P_i \P_j \oplus \P_i \P_j \la 2 \ra \oplus \P_i \P_i \la 1 \ra \ar[rr]^{D} &&  \P_i \P_i \la 1 \ra \oplus \P_i \P_i \la 3 \ra
}
\end{equation*}
where the leftmost column is in degree zero and 
\begin{align*}
A &=
\begin{pmatrix}
1_j 1_i \\
1_j X_i^i \\
X_j^i 1_i
\end{pmatrix}
\ \ \
B = 
\begin{pmatrix}
X_j^i 1_i & 0 & -1_i 1_i \\
0 & X_j^i 1_i & -1_i X_i^i
\end{pmatrix}
\\
C&=
\begin{pmatrix}
1_i 1_j \\
X_i^i 1_j \\
1_i X_j^i
\end{pmatrix}
\ \ \
D = 
\begin{pmatrix}
-1_i X_j^i & 0 & 1_i 1_i \\
0 & -1_i X_j^i & X_i^i 1_i
\end{pmatrix}.
\end{align*}

On the other hand, if $ \la i, j \ra = 0 $ then $ \s_i^{-1}(T_{ij}) $ and $ \s_i^{-1}(T_{ji}) $ are given by 
\begin{equation*}
\xymatrix{
\s_i^{-1}(\P_i \P_j) \ar[dd]^{\s_i^{-1}(T_{ij})} &=&  \P_i \P_j \ar[rr]^{( 1_i 1_j \ \ X_i^i 1_j ) \hspace{.4in}} \ar[dd]^{T_{ij}} && \P_i \P_j \oplus \P_i \P_j \la 2 \ra \ar[dd]^{\begin{pmatrix} T_{ij} & 0 \\ 0 & T_{ij} \end{pmatrix}}\\
\\
\s_i^{-1}(\P_j \P_i) &=&   \P_j \P_i  \ar[rr]^{( 1_j 1_i \ \ 1_j X_i^i ) \hspace{.4in}} && \P_j \P_i \oplus \P_k \P_i \la 2 \ra
}
\end{equation*}
\begin{equation*}
\xymatrix{
\s_i^{-1}(\P_j \P_i) \ar[dd]^{\s_i^{-1}(T_{ji})} &=&  \P_j \P_i \ar[rr]^{( 1_j 1_i \ \ 1_j X_i^i ) \hspace{.4in}} \ar[dd]^{T_{ji}} && \P_j \P_i \oplus \P_j \P_i \la 2 \ra \ar[dd]^{\begin{pmatrix} T_{ji} & 0 \\ 0 & T_{ji} \end{pmatrix}}\\
\\
\s_i^{-1}(\P_i \P_j) &=&  \P_i \P_j  \ar[rr]^{( 1_i 1_j \ \ X_i^i 1_j) \hspace{.4in}} && \P_i \P_j \oplus \P_i \P_j \la 2 \ra
}
\end{equation*}
where the leftmost columns are in cohomological degree zero.

Next, suppose $ \la i, j \ra = -1 = \la i, k \ra $.  Define $ \s_i^{-1}(T_{jk}) $ by 
\begin{equation*}
\xymatrix{
\s_i^{-1}(\P_j \P_k) \ar[dd]^{\s_i^{-1}(T_{jk})} &=&  \P_j \P_k \ar[rrr]^{( X_j^i 1_k \ \ 1_j X_k^i) \hspace{.5in}} \ar[dd]^{T_{jk}} &&& \P_i \P_k \la 1 \ra \oplus \P_j \P_i \la 1 \ra \ar[rrrr]^{\hspace{.5in}( 1_i X_k^i \ \ X_j^i 1_i )} \ar[dd]^{\begin{pmatrix} 0 & T_{ji} \\ T_{ik} & 0 \end{pmatrix}}  &&&& \P_i \P_i \la 2 \ra \ar[dd]^{T_{ii}}  \\
\\
\s_i^{-1}(\P_k \P_j) &=&  \P_k \P_j \ar[rrr]^{( X_k^i 1_j \ \ 1_k X_j^i )\hspace{.5in}} &&& \P_i \P_j \la 1 \ra \oplus \P_k \P_i \la 1 \ra \ar[rrrr]^{\hspace{.5in} ( 1_i X_j^i \ \  X_k^i 1_i)} &&&& \P_i \P_i \la 2 \ra
}
\end{equation*}
where the leftmost column is in cohomological degree zero.

On the other hand, if $ \la i, j \ra = - 1 $ and $ \la i,k \ra = 0 $ then $ \s_i^{-1}(T_{jk}) $ and $ \s_i^{-1}(T_{kj}) $ are defined by
\begin{equation*}
\xymatrix{
\s_i^{-1}(\P_j \P_k) \ar[d]^{\s_i^{-1}(T_{jk})} &=&  \P_j \P_k \ar[r]^{X_j^i 1_k} \ar[d]^{T_{jk}} & \P_i \P_k \la 1 \ra \ar[d]^{T_{ik}}\\
\s_i^{-1}(\P_k \P_j) &=&  \P_k \P_j  \ar[r]^{1_k X_j^i} & \P_k \P_i \la 1 \ra
}
\end{equation*}
\begin{equation*}
\xymatrix{
\s_i^{-1}(\P_k \P_j) \ar[d]^{\s_i^{-1}(T_{kj})} &=&  \P_k \P_j \ar[r]^{1_k X_j^i} \ar[d]^{T_{kj}} & \P_k \P_i \la 1 \ra \ar[d]^{T_{ki}}\\
\s_i^{-1}(\P_j \P_k) &=&  \P_j \P_k  \ar[r]^{X_j^i 1_k} & \P_i \P_k \la 1 \ra
}
\end{equation*}
where both leftmost columns are in cohomological degree zero.

Finally if $ \la i, j \ra = 0 = \la i, k \ra $ then $ \s_i^{-1}(T_{jk}) = T_{jk} $.

\begin{prop}\label{prop:welldefined}
The definitions above give well defined endomorphisms of the 2-category $\Kom(\H)$. 
\end{prop}

\subsection{Some homotopy equivalences} 

Some of the definitions above can be simpified, as we now explain. The reason we do not use these simpler definitions is that in practice they are more difficult to work with when checking the braid relations in the next section. 

The complex $\s_i(\P_i)$ is homotopy equivalent to $\P_i \la -2 \ra [1]$ via the maps $\nu_{\P_i}$ and $\bar{\nu}_{\P_i}$ defined as follows 
\begin{equation*}
\xymatrix{
\s_i(\P_i) \ar[d]_{\nu_{\P_i}} & = & \P_i \la -2 \ra \oplus \P_i \ar[rr]^{\hspace{.5cm} (X_i^i \hspace{.1in} 1_i)} \ar[d]_{(1_i \hspace{.1cm} 0)} & &  \P_i \ar[d] \\
\P_i \la -2 \ra [1] \ar[d]_{\bar{\nu}_{\P_i}} & = & \P_i \la -2 \ra \ar[rr] \ar[d]_{(1_i \hspace{.1cm} -X_i^i)} & & 0 \ar[d] \\
\s_i(\P_i) & = & \P_i \la -2 \ra \oplus \P_i \ar[rr]^{\hspace{.5cm} (X_i^i \hspace{.1in} 1_i)} & &  \P_i
}
\end{equation*}
Clearly $\nu_{\P_i} \bar{\nu}_{\P_i}$ is the identity map. On the other hand, $\bar{\nu}_{\P_i} \nu_{\P_i} $ is homotopic to the identity using $(0 \hspace{.1cm} -1_i): \P_i \rightarrow \P_i \la -2 \ra \oplus \P_i$. 

Likewise, $\s_i^{-1}(\P_i)$ is homotopy equivalent to $\P_i \la 2 \ra [-1]$ via the maps $\zeta_{\P_i}$ and $\bar{\zeta}_{\P_i}$ defined by
\begin{equation*}
\xymatrix{
\s_i^{-1}(\P_i) \ar[d]_{{\zeta}_{\P_i}} & = & \P_i \ar[d] \ar[rr]^{\hspace{-.3cm} (1_i \hspace{.1cm} X_i^i)} & & \P_i \oplus \P_i \la 2 \ra \ar[d]^{(-X_i^i \hspace{.1cm} 1_i)} \\
\P_i \la 2 \ra [-1] \ar[d]_{\bar{\zeta}_{\P_i}} & = & 0 \ar[d] \ar[rr] & & \P_i \la 2 \ra \ar[d]^{(0 \hspace{.1cm} 1_i)} \\
\s_i^{-1}(\P_i) & = & \P_i \ar[rr]^{\hspace{-.3cm}(1_i \hspace{.1cm} X_i^i)} & & \P_i \oplus \P_i \la 2 \ra
}
\end{equation*}

Using these homotopy equivalences one simplify some of the definitions above. In particular, if $\la i,j \ra = -1$ then we get 
\begin{equation*}
\xymatrix{
\nu_{\P_i} \s_i(\P_i \la -1 \ra) \ar[d]_{\s_i(X_i^j) \circ \bar{\nu}_{\P_i}} & = & \P_i \la -3  \ra \ar[r] \ar[d]^{-X_i^i}& 0 \ar[d] \\
\s_i(\P_j) & = & \P_i \la -1 \ra \ar[r]^{X_i^{j}} & \P_{j} 
}
\end{equation*}
\begin{equation*}
\xymatrix{
\s_i(\P_j \la -1 \ra) \ar[d]_{\nu_{\P_i} \circ \s_i(X_j^i)} & = & \P_i \la -2  \ra \ar[r]^{X_i^{j}} \ar[d]^{\epsilon_{ij} 1_i}& \P_{j} \la -1 \ra  \ar[d]  \\
\nu_{\P_i} \s_i(\P_i) & = & \P_i \la -2 \ra \ar[r]^{} & 0
}
\end{equation*}
Moreover, it is easy to check that
\begin{equation*}
(\nu_{\P_i} \nu_{\P_i}) \circ \s_i(T_{ii}) \circ (\bar{\nu}_{\P_i} \bar{\nu}_{\P_i}) = -T_{ii} \colon \P_i \P_i \la -4 \ra [2] \rightarrow \P_i \P_i \la -4 \ra [2].
\end{equation*}
while if $\la i,j \ra = -1$ then 
\begin{equation*}
\xymatrix{
(\nu_{\P_i} \s_i(\P_i)) \s_i (\P_j) \ar[d]_{(1_{\s_i(\P_j)} \nu_{\P_i}) \circ (\s_i(T_{ij})) \circ (\bar{\nu}_{\P_i} 1_{\s_i(\P_j)})} & = & \P_i \P_i \la -3 \ra \ar[r]^{-1_i X_i^j} \ar[d]^{-T_{ii}} & \P_i \P_j \la -2 \ra \ar[d]^{T_{ij}} \\
\s_i(\P_j) (\nu_{\P_i} \s_i(\P_i)) & = & \P_i \P_i \la -3 \ra \ar[r]^{X_i^j 1_i} & \P_j \P_i \la -2 \ra
}
\end{equation*}
\begin{equation*}
\xymatrix{
\s_i(\P_j) (\nu_{\P_i} \s_i(\P_i)) \ar[d]_{(\nu_{\P_i} 1_{\s_i(\P_j)}) \circ (\s_i(T_{ji})) \circ (1_{\s_i(\P_j)} \bar{\nu}_{\P_i})} & = & \P_i \P_i \la -3 \ra \ar[r]^{X_i^j 1_i} \ar[d]^{-T_{ii}} & \P_j \P_i \la -2 \ra \ar[d]^{T_{ji}} \\
(\nu_{\P_i} \s_i(\P_i)) \s_i(\P_j) & = & \P_i \P_i \la -3 \ra \ar[r]^{-1_i X_i^j} & \P_i \P_j \la -2 \ra
}
\end{equation*}
where the rightmost columns are in cohomological degree $-1$. 

Finally, if $\la i,j \ra = 0$ then 
\begin{align*}
1_j \nu_{\P_i} \circ \s_i(T_{ij}) \circ \bar{\nu}_{\P_i}1_j = T_{ij}: & \P_i \P_j \la -2 \ra [1] \rightarrow \P_j \P_i \la -2 \ra [1] \\
\nu_{\P_i} 1_j \circ \s_i(T_{ji}) \circ 1_j \bar{\nu}_{\P_i} = T_{ji}: & \P_j \P_i \la -2 \ra [1] \rightarrow \P_i \P_j \la -2 \ra [1].
\end{align*}

There are similar simplifications involving $\s_i^{-1}$ which we omit. 

\section{Some remarks and conjectures}\label{sec:remarks}

\subsection{The action on 1-morphisms in $\Kom(\H')$}

The action of $\Br(D)$ on $\Kom(\H')$ from section \ref{sec:braidH} was defined explicitly only on 1-morphisms $\P_i$. On some more general 1-morphisms it acts as follows.

\begin{prop}\label{prop:2}
The braid group action from section \ref{sec:braidH} on the 2-category $\Kom(\H')$ acts on $\P_j^{(n)}$ by
$$\s_i(\P_j^{(n)}) \cong
\begin{cases}
\P_i^{(1^n)} \la -2n \ra [n] & \text{ if } i=j \\
\left[\P_i^{(n)} \la -n \ra \rightarrow \P_i^{(n-1)} \P_j \la -n+1 \ra \rightarrow \dots \rightarrow \P_i \P_j^{(n-1)} \la -1 \ra \rightarrow \P_j^{(n)} \right] & \text{ if } \la i,j \ra = -1 \\
\P_j^{(n)} & \text{ if } \la i,j \ra = 0
\end{cases}$$
where the rightmost term is in cohomological degree zero and by
$$\s_i^{-1}(\P_j^{(n)}) \cong
\begin{cases}
\P_i^{(1^n)} \la 2n \ra [-n] & \text{ if } i=j \\
\left[\P_j^{(n)} \rightarrow \P_j^{(n-1)} \P_i \la 1 \ra \rightarrow \dots \rightarrow \P_j \P_i^{(n-1)} \la n-1 \ra \rightarrow \P_i^{(n)} \la n \ra \right] & \text{ if } \la i,j \ra = -1 \\
\P_j^{(n)} & \text{ if } \la i,j \ra = 0
\end{cases}$$
where the leftmost term is in degree zero. The differentials in the complexes are the unique morphisms (up to a scalar) given by the compositions
\begin{align*}
& \P_i^{(r)} \P_j^{(n-r)} \rightarrow \P_i^{(r-1)} \P_i \P_j^{(n-r)} \xrightarrow{IX_i^jI} \P_i^{(r-1)} \P_j \P_j^{(n-r)} \la 1 \ra \rightarrow \P_i^{(r-1)} \P_j^{(n-r+1)} \la 1 \ra \\
& \P_i^{(r)} \P_j^{(n-r)} \rightarrow \P_i^{(r)} \P_j \P_j^{(n-r-1)} \xrightarrow{IX_j^iI} \P_i^{(r)} \P_i \P_j^{(n-r-1)} \la 1 \ra \rightarrow \P_i^{(r+1)} \P_j^{(n-r-1)} \la 1 \ra.
\end{align*}
\end{prop}
\begin{proof}
In order to compute $ \s_i(\P_j^{(n)}) $ one must compute $ \s_i(\P_j^n) $ and determine the image of $ \s_i(e_{(n)}) $ where $ e_{(n)} $ is the trivial idempotent in $ S_n $ acting on $n$ strands colored by $ j $.

The case that $ \la i, j \ra = 0 $ is trivial since $ \s_i(\P_j^n) = \P_j^n $ and $ \s_i(T_{jj}) = T_{jj} $.  Thus the image of $ \s_i(e_{(n)}) $ on $ \P_j^n $ is $ \P_j^{(n)} $ by definition.

If $i=j$ then, up to homotopy, $ \s_i(\P_i^n) \cong \P_i^n \la -2n \ra [n] $ and $ \s_i(e_{(n)}) = e_{(1^n)} $ since $ \s_i(T_{ii}) = -T_{ii} $.  Thus the image of $ \s_i(e_{(n)}) $ on $ \s_i(\P_i^n) $ is $ \P_i^{(1^n)} \la -2n \ra [n] $.

If $\la i, j \ra = -1$ then $\s_i(\P_j) = [\P_i \la -1 \ra \rightarrow \P_j]$ which means that $\s_i(\P_j^n) $ is now a complex. To simplify notation, for a sequence $ {\bf d} = (d_1, \ldots, d_n) $ where each entry is $ i $ or $ j $ we denote $ \P_{\bf d} = \P_{d_1} \cdots \P_{d_n} $. Then the term in $\s_i(\P_j^n)$ lying in cohomological degree $-r$ (where $r \ge 0$) is $\bigoplus_{\bf d} \P_{\bf d} \la -r \ra$ where $r$ of the entries of $ {\bf d} $ are $i$ and $n-r$ are $j$. 

Since $S_n$ permutes the entries of $ {\bf d} $ we may consider the subgroup $ S_{\bf d} \subset S_n $ which stabilizes ${\bf d}$.  Let ${\bf d'}$ be another sequence where $ r $ of the entries are $ i $ and $ n-r $ are $ j $. Then, using the definition of $\sigma_i(T_{jj})$ from section \ref{sec:action4}, the 2-morphism $\s_i(e_{(n)}) $ induces a map $ \P_{\bf d} \la -r \ra \rightarrow \P_{\bf d'} \la -r \ra $ which (up to a nonzero scalar) is the sum over all elements in the coset of $ S_n / S_{\bf d} $ which transforms $ {\bf d} $ to $ {\bf d'} $.

Since for any ${\bf d,d'}$ as above there is an isomorphism $ \phi \colon \P_{\bf d} \rightarrow \P_{\bf d'} $ such that $ \s_i(e_{(n)})_{|\P_{\bf d}} =  \s_i(e_{(n)})_{|\P_{\bf d'}} \circ \phi $ it suffices to consider the compute the image of $ \s_i(e_{(n)}) $ on $ \P_{({\bf i},{\bf j})} $ where $({\bf i}, {\bf j}) := (\underbrace{i, \ldots, i,}_r \underbrace{j, \ldots, j}_{n-r})$.

So consider now $\s_i(e_{(n)})_{|\P_{\bb}}: \P_{\bb} \rightarrow \bigoplus_{\bf d} \P_{\bf d}$. We must show that the image is isomorphic to $\P_i^{(r)} \P_j^{(n-r)}$. Let $ w_{{\bb}, {\bf d}} $ be a minimal length representative in this coset. The component of this map $\P_{\bb} \rightarrow \P_{\bf d}$ is a sum over all elements in the coset of $ S_n / S_r \times S_{n-r} $ which transforms $ {\bb} $ to $ {\bf d} $. 

Consider the composition $ \P_{\bb} \rightarrow \bigoplus_{\bf d} \P_{\bf d} \rightarrow \bigoplus_{\bf d} \P_{\bb} $ where the first map is the restriction of $ \s_i(e_{(n)}) $ and the second map is the isomorphism $ \bigoplus_{\bf d} w_{{\bb}, {\bf d}}^{-1} $.  The composition is the diagonal map where each entry is the sum over all elements in $S_r \times S_{n-r}$. Thus the image of $ \s_i(e_{(n)}) $ in $ \P_{\bb} $ is $ \P_i^{(r)} \P_j^{(n-r)}$.


For example, take $ r=2, n=3 $ so that $ {\bb} = (i,i,j) $.  
Then the restriction of $ \s_i(e_{(3)}) $ to  $ \P_{(i,i,j)} $ is the map
\begin{equation*}
\begin{pmatrix}
1+s_1 \\
s_2 + s_2 s_1 \\ 
s_1 s_2 s_1 + s_1 s_2
\end{pmatrix}
\colon \P_{(i,i,j)} \rightarrow \P_{(i,i,j)} \oplus \P_{(i,j,i)} \oplus \P_{(j,i,i)}
\end{equation*}
where the $ s_1,s_2 $ are the standard transpositions. Composing this map with 
\begin{equation*}
\bigoplus_{\bf d} w_{{\bb}, {\bf d}}^{-1} = 
\begin{pmatrix}
1 & & \\
& s_2 & \\
& & s_2 s_1
\end{pmatrix}
\colon \P_{(i,i,j)} \oplus \P_{(i,j,i)} \oplus \P_{(j,i,i)} \rightarrow \P_{(i,i,j)} \oplus \P_{(i,i,j)} \oplus \P_{(i,i,j)}
\end{equation*}
we obtain
\begin{equation*}
\begin{pmatrix}
1+s_1 \\
1+s_1 \\
1+s_1
\end{pmatrix}
\colon \P_{(i,i,j)} \rightarrow \P_{(i,i,j)} \oplus \P_{(i,i,j)} \oplus \P_{(i,i,j)}.
\end{equation*}

The computation of $ \s_i^{-1}(\P_j^{(n)}) $ is similar so we omit it.
\end{proof}

\begin{Remark}\label{rem:2}
It was shown in \cite{CL1} that $\K_{Fock}$ categorifies $V_{Fock}$ (the Fock space). This means that there exists an isomorphism 
$$\Phi: K_0(\Kom(\K_{Fock})) = K_0(\K_{Fock}) \rightarrow V_{Fock}$$
which takes $\P_i^{(n)} \1_0 \mapsto P_i^{(n)}$ in the notation of section \ref{subsec:braidfock}. Subsequently, Proposition \ref{prop:2} induces the action of $\Br(D)$ on $V_{Fock}$ described in Proposition \ref{prop:fockspace}.
\end{Remark}

\begin{Remark} 
Consider the zig-zag algebra $B$ of Dynkin type $A_k$. In Section 4 of \cite{KS}, a complex of finitely generated, projective $B$-modules is associated to any ``admissible'' curve in the $k$-punctured disk. Equivalently, this is a complex of $\P$'s in $\Kom(\H')$ (for example $\P_1 \la -1 \ra \rightarrow \P_2$). It would be interesting to generalize that result as follows: to any ``admissible'' $n$-tuple of curves in the $k$-punctured disk, associate a complex of elements in $\H'$ where each term is a sum of a product of exactly $n$ $\P$'s (for example, the complexes appearing in the statement of Proposition \ref{prop:2}). 
\end{Remark}

\subsection{A conjectural intertwiner}\label{sec:intertwiner}

At this point we have two braid group actions -- one is via the complexes $\T_i$ and the other is directly on the 2-category $\Kom(\H)$. We conjecture that these two actions are related as follows.

\begin{conj}\label{conj:intertwiner}
Consider a 2-representation $\K$ of $\H$ and denote by $\R$ an arbitrary 1-morphism in $\Kom(\H)$. Then $ \s_i(\R) \circ \T_i  \cong \T_i \circ \R $.
\end{conj}

In particular, this means that the braid group action on $\Kom(\H)$ from section \ref{sec:braidH} is just conjugation using the complexes $\T_i$. 

\subsection{Vertex operators and braid groups}\label{sec:vertexops}

In this section we suppose that our Dynkin diagram $D$ is of affine type (and still simply laced). Denote by $\g$ the associated affine Lie algebra. In \cite{CLa} we defined what it means to have a 2-representation of $\g$. Roughly, this consists of a 2-category where the objects are indexed by weights of $\g$, 1-morphisms are generated by $\E_i^{(r)}$ and $\F_i^{(r)}$ where $i \in I, r \in \N$ and there are various 2-morphisms with relations. This definition is analogous to the one from this paper for 2-representations of $\h$.

Suppose $\K$ is a integrable 2-representation of $\h$. In \cite{CL2} we showed that $\Kom(\K)$ can be given the structure of a 2-representation of $\g$. This categorifies the Frenkel-Kac-Segal vertex operator construction. Roughly, we did the following. 
\begin{itemize}
\item We defined $\1_\l \mapsto \1_n$ if $\l = w \cdot \Lambda_0 - n \delta$ for some Weyl element $w$ and $\1_\l \mapsto 0$ otherwise. Here $\l$ is a weight, $\Lambda_0$ is the fundamental weight corresponding to the affine node and $\delta$ is the imaginary root.
\item We mapped $\E_i,\F_i$ as follows
\begin{align}
\label{eq:vertex1}
\E_i \1_\l &\mapsto \left[ \dots \rightarrow \P_i^{(l)} \Q_i^{(1^{k+l})} \la -l \ra \rightarrow \dots \rightarrow \P_i \Q_i^{(1^{k+1})} \la -1 \ra \rightarrow \Q_i^{(1^{k})} \right] \\
\label{eq:vertex2}
\1_\l \F_i &\mapsto \left[ \P_i^{(1^{k})} \rightarrow \P_i^{(1^{k+1})} \Q_i \la 1 \ra \rightarrow \dots \rightarrow \P_i^{(1^{k+l})} \Q_i^{(l)} \la l \ra \rightarrow \dots \right]
\end{align}
if $k := \la \l, \alpha_i \ra + 1 \ge 0$ (and to similar complexes if $k < 0$).
\end{itemize}
We argued that this action extends to give a 2-representation of $\g$. This means that there also exist complexes for divided power $\E_i^{(r)} \1_\l$ and $\1_\l \F_i^{(r)}$. We gave the following conjectural explicit description of these complexes (proven when $r=1,2$). Although we did not identify these complexes explicitly (apart from the cases $r=1,2$) we did conjecture that in general, if $k := \la \l, \alpha_i \ra + r \ge 0$, then 
\begin{equation}\label{eq:E}
\E_i^{(r)} \1_\l :=
\left[ \dots \rightarrow \bigoplus_{w(\mu) \le r, |\mu|=l} \P_i^{(\mu^t)} \Q_i^{(r^k, \mu)} \la -l \ra \rightarrow \dots \rightarrow \P_i \Q_i^{(r^k,1)} \la -1 \ra \rightarrow \Q_i^{(r^k)} \right] \la - \binom{r}{2} \ra [\binom{r}{2}]. 
\end{equation}
for certain explicit differentials. Here the direct sum is over all partitions $\mu$ of size $|\mu|$ which fit in a box of width $r$ ($w(\mu)$ denotes the width of $\mu$). We also conjectured similar formulas if $k \le 0$ and likewise for $\1_\l \F_i^{(r)}$. 

\subsubsection{Associated braid group actions}
Given any 2-representation of $\g$, we considered in \cite{CK2} the Rickard complex defined by
\begin{equation}\label{eq:T}
\sT_i \1_\l := \left[ \dots \rightarrow \E_i^{(-\la \l, \alpha_i \ra + s)} \F_i^{(s)} \la -s \ra \1_\l \rightarrow \dots \rightarrow \E_i^{(- \la \l,\alpha_i \ra + 1)} \F_i \la -1 \ra \1_\l \rightarrow \E_i^{(- \la \l,\alpha_i \ra)} \1_\l \right]
\end{equation}
if $\la \l,\alpha_i \ra \le 0$ (and similarly if $\la \l,\alpha_i \ra \ge 0$). We then showed \cite[Theorem 2.10]{CK2} that these complexes satisfy the braid relations in $\Br(D)$. Notice that the domain and range of $\sT_i \1_\l$ are given by
$$\sT_i \1_\l: \l \rightarrow s_i \cdot \l \ \ \text{ where } \ \ s_i \cdot \l = \l - \la \l, \alpha_i \ra \alpha_i.$$
If $\1_\l \mapsto \1_n$ under the map from \cite{CL2} then it is easy to check that $\1_{s_i \cdot \l} \mapsto \1_n$. Thus, if we compose the complexes for $\sT_i$ from (\ref{eq:T}) with those for $\E$'s and $\F$'s (from (\ref{eq:vertex1}) and (\ref{eq:vertex2})) then we obtain complexes $\sT_i \in \Kom(\K)$ with domain and range $n$. 

{\bf Example 1.} Suppose $\1_\l \mapsto \1_n$ under the map from \cite{CL2} with $\la \l, \alpha_i \ra = -n$. Then, the complex from (\ref{eq:E}) with $k=0$ and $r=n$ gives the following expression for $\E_i^{(n)} \1_\l$ 
\begin{equation}\label{eq:E'}
\left[ \dots \rightarrow \bigoplus_{w(\mu) \le n, |\mu|=l} \P_i^{(\mu^t)} \Q_i^{(\mu)} \la -l \ra \1_n \rightarrow \dots \rightarrow \P_i \Q_i \la -1 \ra \1_n \rightarrow \1_n \right] \la - \binom{n}{2} \ra [\binom{n}{2}].
\end{equation}
Notice that the terms in (\ref{eq:E'}) are zero if $|\mu| > n$ so the extra condition that $w(\mu) \le n$ is not necessary. Subsequently, $\E_i^{(n)} \1_\l$ is the same as our complex $\T_i \1_n$.  On the other hand, it is not difficult to check that in this case $\F_i^{(s)} \1_\l = 0$ for any $s > 0$. Thus, the expression in (\ref{eq:T}) simplifies to give $\sT_i \1_\l = \E_i^{(n)} \1_\l$.  Thus, $\sT_i \1_\l = \T_i \1_n$ and, using \cite{CK2}, we recover the braiding of the $\T_i \1_n$ (Theorem \ref{thm:main1}). 

{\bf Example 2.} Suppose $\la \l, \alpha_i \ra = 1$ and that, under the map in \cite{CL2}, $\1_\l \mapsto \1_2$. Then
$$\sT_i \1_\l = \left[ \F_i^{(2)} \E_i \la -1 \ra \1_\l \rightarrow \F_i \1_\l \right]: \l \rightarrow s_i \cdot \l.$$
Using the definitions in \cite{CL2}, one can check that  
\begin{align*}
\F_i \1_\l &\mapsto \left[ \1_2 \rightarrow \P_i \Q_i \la 1 \ra \1_2 \rightarrow \P_i^{(1^2)} \Q_i^{(2)} \la 2 \ra \1_2 \right] \\
\F_i^{(2)} \1_{\l + \alpha_i} &\mapsto \P_i^{(2)} \la 1 \ra [-1] \1_0 \ \ \text{ and } \ \ \E_i \1_\l \mapsto \Q_i^{(1^2)} \1_2.
\end{align*}
Combining this together gives that
$$\P_i^{(1^2)} \Q_i^{(2)} [-1] \1_2 \rightarrow \left[ \1_2 \rightarrow \P_i \Q_i \la 1 \ra \1_2 \rightarrow \P_i^{(2)} \Q_i^{(1^2)} \la 2 \ra \1_2 \right].$$
This collapses to give a complex of the form
\begin{equation}\label{eq:eg2}
\left[ \P_i^{(1^2)} \Q_i^{(2)} \1_2 \oplus \1_2 \rightarrow \P_i \Q_i \la 1 \ra \1_2 \rightarrow \P_i^{(2)} \Q_i^{(1^2)} \la 2 \ra \1_2 \right].
\end{equation}
Notice that this complex, is {\em not} of the form of $\T_i \1_2$. 

\subsubsection{In conclusion}
Using \cite{CK2}, and assuming the conjectural expressions for $\E_i^{(r)}$ and $\F_i^{(r)}$ from \cite{CL2} (like the one in (\ref{eq:E})) we recover the main result in this paper. On the other hand, as Example 2 above illustrates, the full categorical action from \cite{CL2} and the braid group action it induces via \cite{CK2} gives us a larger collection of complexes in $\Kom(\H)$ which satisfy the braid relations ((\ref{eq:eg2}) is an example of one such complex). It would be interesting to explicitly identify all these complexes in $\Kom(\H)$ directly. 

\appendix

\section{Proof of Proposition \ref{prop:welldefined}}

What one needs to check is that the image of any two equivalent 2-morphisms ({\it i.e.} related by some 2-relation) under any $\s_i^{\pm 1}$ are identical. 

There are many 2-relations so we will not check all of them in this paper. We illustrate by checking one of the most difficult relations, namely 
$$\s_i(T_{jj} 1_j) \circ \s_i(1_j T_{jj}) \circ \s_i (T_{jj} 1_j) = \s_i(1_j T_{jj}) \circ \s_i (T_{jj} 1_j) \circ \s_i(1_j T_{jj})$$
whenever $\la i,j \ra = -1$. 

By direct computation, $\s_i(T_{jj} 1_j) \circ \s_i(1_j T_{jj}) \circ \s_i(T_{jj} 1_j)$ is the following map of complexes:
\begin{equation*}
\xymatrix{
\P_i \P_i \P_i \la -3 \ra \ar[r]^{a} \ar[dd]^{A}  & {\begin{matrix}  \P_i \P_i \P_j \la -2 \ra \\ \oplus \\  \P_i \P_j \P_i \la -2 \ra \\ \oplus \\  \P_j \P_i \P_i \la -2 \ra \end{matrix}} \ar[r]^{b} \ar[dd]^{B}  & {\begin{matrix} \P_i \P_j \P_j \la -1 \ra \\ \oplus \\  \P_j \P_i \P_j \la -1 \ra \\ \oplus \\  \P_j \P_j \P_i \la -1 \ra  \end{matrix}} \ar[r]^{c} \ar[dd]^{C} & \P_j \P_j \P_j  \ar[dd]^{D} \\
\\
\P_i \P_i \P_i \la -3 \ra \ar[r]^{d}  & {\begin{matrix}  \P_i \P_i \P_j \la -2 \ra \\ \oplus \\ \P_i \P_j \P_i \la -2 \ra \\ \oplus \\ \P_j \P_i \P_i \la -2 \ra  \end{matrix}} \ar[r]^{e} & {\begin{matrix} \P_i \P_j \P_j \la -1 \ra \\ \oplus \\ \P_j \P_i \P_j \la -1 \ra  
\\ \oplus \\ \P_j \P_j \P_i \la -1 \ra \end{matrix}} \ar[r]^{f}  & \P_j \P_j \P_j  
}
\end{equation*}
where 
\begin{equation*}
a = 
\begin{pmatrix}
1_i 1_i X_i^j \\
1_i X_i^j 1_i \\
X_i^j 1_i 1_i \\
\end{pmatrix}
\hspace{.2in}
b = 
\begin{pmatrix}
1_i X_i^j 1_j & 1_i 1_j X_i^j & 0 \\
X_i^j & 0 & 1_j 1_i X_i^j \\
0 & X_i^j 1_j 1_i & 1_j X_i^j 1_i
\end{pmatrix}
\hspace{.2in}
c = 
\begin{pmatrix}
X_i^j 1_j 1_j & 1_j X_i^j 1_j & 1_j 1_j X_i^j
\end{pmatrix}
\end{equation*}
\begin{equation*}
d = 
\begin{pmatrix}
1_i X_i^j 1_i \\
1_i 1_i X_i^j \\
X_i^j 1_i 1_i \\
\end{pmatrix}
\hspace{.2in}
e = 
\begin{pmatrix}
1_i 1_j X_i^j & 1_i X_i^j 1_j & 0 \\
X_i^j & 0 & 1_j X_i^j 1_i \\
0 & X_i^j 1_i 1_j & 1_j 1_i  X_i^j
\end{pmatrix}
\hspace{.2in}
f = 
\begin{pmatrix}
X_i^j 1_j 1_j & 1_j 1_j X_i^j & 1_j X_i^j 1_j
\end{pmatrix}
\end{equation*}
\begin{equation*}
A=
\begin{pmatrix}
(T_{ii} 1_i) \circ (1_i T_{ii}) \circ (T_{ii} 1_i)
\end{pmatrix}
\hspace{3.4in}
\end{equation*}
\begin{equation*}
B = 
\begin{pmatrix}
0 & (T_{ji} 1_i) \circ (1_j T_{ii}) \circ (T_{ij} 1_i) & 0 \\
0 & 0 & (T_{ii} 1_j) \circ (1_i T_{ji}) \circ (T_{ji} 1_i) \\
(T_{ij} 1_i) \circ (1_i T_{ij}) \circ (T_{ii} 1_j) & 0 & 0
\end{pmatrix}
\end{equation*}
\begin{equation*}
C = 
\begin{pmatrix}
0 & 0 & (T_{ji} 1_j) \circ (1_j T_{ji}) \circ (T_{jj} 1_i) \\
(T_{jj} 1_i) \circ (1_j T_{ij}) \circ (T_{ij} 1_j) & 0 & 0 \\
0 & (T_{ij} 1_j) \circ (1_i T_{jj}) \circ (T_{ji} 1_j) & 0
\end{pmatrix}
\end{equation*}
\begin{equation*}
D=
\begin{pmatrix}
(T_{jj} 1_j) \circ (1_j T_{jj}) \circ (T_{jj} 1_j)
\end{pmatrix}.
\hspace{3.3in}
\end{equation*}
Similarly, 
$  \s_i(1_j T_{jj}) \circ \s_i(T_{jj} 1_j) \circ \s_i(1_j T_{jj}) $ is a map of complexes:
\begin{equation*}
\xymatrix{
\P_i \P_i \P_i \la -3 \ra \ar[r]^{a} \ar[dd]^{A'}  & {\begin{matrix}  \P_i \P_i \P_j \la -2 \ra \\ \oplus \\  \P_i \P_j \P_i \la -2 \ra \\ \oplus \\  \P_j \P_i \P_i \la -2 \ra \end{matrix}} \ar[r]^{b} \ar[dd]^{B'}  & {\begin{matrix} \P_i \P_j \P_j \la -1 \ra \\ \oplus \\  \P_j \P_i \P_j \la -1 \ra \\ \oplus \\  \P_j \P_j \P_i \la -1 \ra  \end{matrix}} \ar[r]^{c} \ar[dd]^{C'} & \P_j \P_j \P_j  \ar[dd]^{D'} \\
\\
\P_i \P_i \P_i \la -3 \ra \ar[r]^{d}  & {\begin{matrix}  \P_i \P_i \P_j \la -2 \ra \\ \oplus \\ \P_i \P_j \P_i \la -2 \ra \\ \oplus \\ \P_j \P_i \P_i \la -2 \ra  \end{matrix}} \ar[r]^{e} & {\begin{matrix} \P_i \P_j \P_j \la -1 \ra \\ \oplus \\ \P_j \P_i \P_j \la -1 \ra  
\\ \oplus \\ \P_j \P_j \P_i \la -1 \ra \end{matrix}} \ar[r]^{f}  & \P_j \P_j \P_j  
}
\end{equation*}
where
\begin{equation*}
A'=
\begin{pmatrix}
(1_i T_{ii}) \circ (T_{ii} 1_i) \circ (1_i T_{ii})
\end{pmatrix}
\hspace{3.4in}
\end{equation*}
\begin{equation*}
B' = 
\begin{pmatrix}
0 & (1_i T_{ij}) \circ (T_{ii} 1_j) \circ (1_i T_{ji}) & 0 \\
0 & 0 & (1_i T_{ji}) \circ (T_{ji} 1_i) \circ (1_j T_{ii}) \\
(1_j T_{ii}) \circ (T_{ij} 1_i) \circ (1_i T_{ij}) & 0 & 0
\end{pmatrix}
\end{equation*}
\begin{equation*}
C' = 
\begin{pmatrix}
0 & 0 & (1_i T_{jj}) \circ (T_{ji} 1_j) \circ (1_j T_{ji}) \\
(1_j T_{ij}) \circ (T_{ij} 1_j) \circ (1_i T_{jj}) & 0 & 0 \\
0 & (1_j T_{ji}) \circ (T_{jj} 1_i) \circ (1_j T_{ij}) & 0
\end{pmatrix}
\end{equation*}
\begin{equation*}
D'=
\begin{pmatrix}
(1_j T_{jj}) \circ (T_{jj} 1_j) \circ (1_j T_{jj})
\end{pmatrix}.
\hspace{3.3in}
\end{equation*}
Equalities of matrices $ A=A', B=B', C=C', D=D' $ follow from the three strand relation in the category.

\section{Proof of Theorem \ref{thm:main2}}

To prove Theorem \ref{thm:main2} one needs to fix isomorphisms $\s_i \s_i^{-1} M \rightarrow M$, $\s_i^{-1} \s_i^{} M \rightarrow M$, $ \s_i \s_{i+1} \s_i M \rightarrow \s_{i+1} \s_i \s_{i+1} M $ on all generating 1-morphisms of $\Kom(\H')$ and then show that these isomorphisms are natural with respect to all generating 2-morphisms. 

\subsubsection{Reidemeister 2 relations on 1-morphisms}
 
Recall the homotopy equivalences 
$$ \nu_{\P_i} : \s_i(\P_i) \xrightarrow{\sim} \P_i \la -2 \ra [1] \ \ \text{ and }  \ \ \zeta_{\P_i}: \s_i^{-1}(\P_i) \xrightarrow{\sim} \P_i \la 2 \ra [-1].$$ 
We can use these to define isomorphisms 
$$\s_i^{-1} \s_i(\P_i) \xrightarrow{\s_i^{-1}(\nu_{\P_i})} \s_i^{-1}(\P_i \la -2 \ra [1]) \xrightarrow{\zeta_{\P_i}} \P_i \ \ \text{ and } \ \ \s_i \s_i^{-1}(\P_i) \xrightarrow{\s_i(\zeta_{\P_i})} \s_i(\P_i \la 2 \ra [-1]) \xrightarrow{\nu_{\P_i}} \P_i.$$

On the other hand, if $\la i,j \ra = -1$ then we use the following maps
\begin{equation*}
\xymatrix{
\s_i^{-1} \s_i (\P_j) \ar[d] & = & \P_i \la -1 \ra \ar[d] \ar[rr]^{\hspace{-1cm} (X_i^j \ \ 1 \ \ X_i^i)} & & \P_j \oplus \P_i \la -1 \ra \oplus \P_i \la 1 \ra \ar[d]^{(1 \ \ -X_i^j \ \ 0)} \ar[rr]^{\hspace{1cm} (X_j^i \ \ 0 \ \ -\epsilon_{ij})} & & \P_i \la 1 \ra \ar[d] \\
\P_j \ar[d] & = & 0 \ar[d] \ar[rr] && \P_j \ar[rr] \ar[d]^{(1 \ \ 0 \ \ \ep_{ij} X_j^i)} && 0 \ar[d] \\
\s_i^{-1} \s_i (\P_j) & = & \P_i \la -1 \ra \ar[rr]^{\hspace{-1cm} (X_i^j \ \ 1 \ \ X_i^i)} & & \P_j \oplus \P_i \la -1 \ra \oplus \P_i \la 1 \ra \ar[rr]^{\hspace{1cm} (X_j^i \ \ 0 \ \ -\epsilon_{ij})} & & \P_i \la 1 \ra 
}
\end{equation*}
where the rightmost column is in cohomological degree one. The vertical composition is 
\begin{equation*}
\xymatrix{
\P_i \la -1 \ra \ar[rr]^{} \ar[d]^{0} & & \P_j \oplus \P_i \la -1 \ra \oplus \P_i \la 1 \ra \ar[rr]^{} \ar[d]^{\gamma_0} \ar[lld]_{D_0}  & & \P_i \la 1 \ra \ar[d]^{0} \ar[lld]_{D_1}\\
\P_i \la -1 \ra \ar[rr]^{} & & \P_j \oplus \P_i \la -1 \ra \oplus \P_i \la 1 \ra \ar[rr]^{} & & \P_i \la 1 \ra 
}
\end{equation*}
where $\gamma_0 = \begin{pmatrix} 1 & -X_i^j & 0 \\ 0 & 0 & 0 \\ \epsilon_{ij} X_i^j & -X_i^i & 0 \end{pmatrix}$. Here $D_0 = (0 \ -1 \ 0)$ and $D_1 = (0 \ 0 \ \epsilon_{ij})$ give a homotopy between this composition and the identity map.

\subsubsection{Reidemeister 3 relations on 1-morphisms}

\begin{prop}
\label{braidp_i}
For $ \la i, j \ra = -1$
there is an isomorphism
\begin{equation*}
\gamma_{\P_i} \colon \s_i \s_{j} \s_i(\P_i) \rightarrow \s_{j} \s_i \s_{j}(\P_i).
\end{equation*}
\end{prop}

\begin{proof}
One checks that the map of complexes $ \beta \colon \s_i \s_{j}(\P_i) \rightarrow \P_j \la -1 \ra $ given by
\begin{equation*}
\xymatrix{
\P_i \la -2 \ra \ar[rr]^{\begin{pmatrix}
-\epsilon_{ij} \\
0 \\ 
X_i^j
\end{pmatrix}\hspace{.4in}} \ar[d] & & \P_i \la -2 \ra \oplus \P_i \la 1 \ra \oplus \P_j \la -1 \ra  \ar[rr]^{\hspace{.7in} \begin{pmatrix}
X_i^i & 1 & X_j^i
\end{pmatrix}} \ar[d]^{\begin{pmatrix}
\epsilon_{ij} X_i^j & 0 & 1
\end{pmatrix}} & & \P_i  \ar[d] \\
0 \ar[rr] & & \P_j \ar[rr] & & 0
}
\end{equation*}
is a homotopy equivalence with inverse map $ \bar{\beta} $.
Then $ \beta \circ \s_i \s_j(\nu_{\P_i}) \colon \s_i \s_j \s_i(\P_i) \rightarrow \P_j \la -3 \ra [2] $ is a homotopy equivalence.

Similarly $ \nu_{\P_{j}} \circ \s_{j}(\beta_{\P_{j}}) \colon \s_{j} \s_i \s_{j}(\P_i) \rightarrow \P_{j} \la -3 \ra [-2] $ is a homotopy equivalence.
Finally we define
\begin{equation*}
\gamma_{\P_i} = -\s_{i+1}(\bar{\beta}_{\P_{i+1}}) \circ \bar{\nu}_{\P_{i+1}} \circ \beta_{\P_{i+1}} \circ \s_i \s_{i+1} (\nu_{\P_i}).
\end{equation*}
\end{proof}

\begin{prop}
For $ \la i, j \ra = -1$
there is an isomorphism
\begin{equation*}
\gamma_{\P_{j}} \colon \s_i \s_{j} \s_i(\P_{j}) \rightarrow \s_{j} \s_i \s_{j}(\P_{j}).
\end{equation*}
\end{prop}

\begin{proof}
The proof of this is similar to that of Proposition ~\ref{braidp_i}. 
\end{proof}

Next, if $\la i,j \ra = -1 = \la j,k \ra$ and $\la i,k \ra =0$ one needs to write down a homotopy equivalence $\s_i \s_j \s_i(\P_k) \xrightarrow{\sim} \s_j \s_i \s_j(\P_k)$. A direct calculation shows that 
\begin{align*}
\s_i \s_j \s_i (\P_k) &= \left[ \P_i \la -2 \ra \xrightarrow{X_i^j} \P_j \la -1 \ra \xrightarrow{X_j^k} \P_k \right] \\
\s_j \s_i \s_j (\P_k) &= \left[ \P_j \la -3 \ra \xrightarrow{A} \P_j \la -3 \ra \oplus \P_j \la -1 \ra \oplus \P_i \la -2 \ra \xrightarrow{B} \P_j \la -1 \ra \oplus \P_j \la -1 \ra \xrightarrow{C} \P_k \right]
\end{align*}
where 
\begin{equation*}
A = \begin{pmatrix} 1 & 0 & X_j^i \end{pmatrix}
\hspace{.2in}
B = \begin{pmatrix} 0 & -1 & 0 \\ X_j^j & 1 & X_i^j \end{pmatrix}
\hspace{.2in}
C = \begin{pmatrix} X_j^k & X_j^k \end{pmatrix}.
\end{equation*}
It is not difficult to check that the following maps are homotopy equivalences:
\begin{equation*}
\xymatrix{
\s_i \s_j \s_i (\P_k) \ar[d]^{\phi} & = & \P_j \la -3 \ra \ar[r]^-{A} \ar[d] & \P_j \la -3 \ra \oplus \P_j \la -1 \ra \oplus \P_i \la -2 \ra \ar[r]^-{B} \ar[d]^{(-X_j^i \ \ 0 \ \ 1)} & \P_j \la -1 \ra \oplus \P_j \la -1 \ra \ar[r]^-{C} \ar[d]^{(1 \ \ 1)} & \P_k \ar[d]^{1} \\
\s_j \s_i \s_j(\P_k) \ar[d]^{\bar{\phi}} & = & 0 \ar[r] \ar[d] & \P_i \la -2 \ra \ar[r]^-{X_i^j} \ar[d]^{(0 \ \ -X_i^j \ \ 1)} & \P_j \la -1 \ra \ar[r]^-{X_j^k} \ar[d]^{(1 \ \ 0)} & \P_k \ar[d]^{1} \\
\s_i \s_j \s_i (\P_k) & = & \P_j \la -3 \ra \ar[r]^-{A} & \P_j \la -3 \ra \oplus \P_j \la -1 \ra \oplus \P_i \la -2 \ra \ar[r]^-{B} & \P_j \la -1 \ra \oplus \P_j \la -1 \ra \ar[r]^-{C} & \P_k
}
\end{equation*}

Finally, if $i,j,k$ form a triangle then there is a more complicated homotopy equivalence $\s_i \s_j \s_i (\P_k) \xrightarrow{\sim} \s_j \s_i \s_j (\P_k)$ which we omit (the interested reader can contact the authors for more details). 

\subsubsection{Reidemeister 2 relations on 2-morphisms} 

First, assuming $\la j,k \ra = -1$, one needs to check that the following diagrams commute
\begin{equation*}
\xymatrix{
\s_i^{-1} \s_i (\P_k \la -1 \ra) \ar[rr]^{ \s_{i}^{-1} \s_i (X_k^j)} \ar[d] &  & \s_i^{-1} \s_i (\P_j) \ar[d] \\
\P_k \la -1 \ra \ar[rr]^{X_k^j} & & \P_j
}
\hspace{1in}
\xymatrix{
\s_i^{} \s_i^{-1} (\P_k \la -1 \ra) \ar[rr]^{ \s_{i}^{} \s_i^{-1} (X_k^j)} \ar[d] & & \s_i^{} \s_i^{-1} (\P_j) \ar[d] \\
\P_k \la -1 \ra \ar[rr]^{X_k^j} & & \P_j.
}
\end{equation*}
where the vertical maps are those from the previous section. 

Next, for any $i,j,k$, one needs to check the commutativity of the following squares
\begin{equation*}
\xymatrix{
\s_i^{-1} \s_i (\P_j \P_k) \ar[rr]^{\s_i^{-1} \s_i (T_{jk})} \ar[d] && \s_i^{-1} \s_i (\P_j \P_k) \ar[d]  \\
\P_j \P_k \ar[rr]^{T_{jk}} && \P_j \P_k
}
\hspace{1in}
\xymatrix{
\s_i^{} \s_i^{-1} (\P_j \P_k) \ar[rr]^{\s_i^{} \s_i^{-1} (T_{jk})} \ar[d] && \s_i^{} \s_i^{-1} (\P_j \P_k) \ar[d]  \\
\P_j \P_k \ar[rr]^{T_{jk}} && \P_j \P_k. }
\end{equation*}

\subsubsection{Reidemeister 3 relations on 2-morphisms}

First, assuming $ \la i,j \ra =-1$ and $\la l,k \ra = -1$, one needs to check that the following diagram commutes
\begin{equation*}
\xymatrix{
\s_i \s_{j} \s_i (\P_{k} \la -1 \ra) \ar[rrr]^{\s_i \s_{i+1} \s_i (X_{k}^{l})} \ar[d]&  & & \s_i \s_{j} \s_i (\P_{l}) \ar[d]\\
\s_{j} \s_{i} \s_{j} (\P_{k} \la -1 \ra) \ar[rrr]^{\s_{j} \s_{i} \s_{j} (X_{k}^{l})}  & & & \s_{j} \s_{i} \s_{j} (\P_{l}) 
}.
\end{equation*}

Next, for any $i,j,k,l$ with $ \la i,j \ra =-1$, one needs to show that the following diagram commutes
\begin{equation*}
\xymatrix{
\s_i \s_{j} \s_i (\P_{l} \P_k) \ar[rrr]^{\s_i \s_{j} \s_i (T_{lk})} \ar[d]^{\gamma_{\P_l} \gamma_{\P_k}} &  & & \s_i \s_{j} \s_i (\P_{k} \P_l) \ar[d]^{\gamma_{\P_k} \gamma_{\P_l}}\\
\s_{j} \s_{i} \s_{j} (\P_{l} \P_k) \ar[rrr]^{\s_{j} \s_{i} \s_{j} (T_{lk})}  & & & \s_{j} \s_{i} \s_{j} (\P_{k} \P_l) 
}.
\end{equation*}

Checking that all these diagrams commute breaks up into many cases. Each case, though not difficult, is a bit tedious (the interested reader can contact the authors for more details about these calculations).

\end{document}